\documentclass[11pt,letterpaper,reqno, english]{amsart}
\usepackage{url,mathptmx}
\usepackage{amsmath,amstext,amsfonts,amssymb,amsbsy,mathrsfs}
\usepackage{amssymb,amsthm, fullpage, color}
\usepackage{latexsym}
\usepackage{array}
\usepackage{enumerate}
\usepackage{amsthm}
\usepackage{amscd}
\usepackage{multirow}
\usepackage{babel}
\usepackage{enumitem}

\newcommand{\lb}{\left\langle}
\newcommand{\rb}{\right\rangle}
\newcommand{\ident}{\equiv}
\newcommand{\abs}[1]{\left|#1\right|}
\newcommand{\norm}[1]{\left\| #1 \right\|}
\newcommand{\union}{\cup}
\newcommand{\intersect}{\cap}
\renewcommand{\Hat}{\widehat}

\newcommand{\bb}{\mathbb}

\newcommand{\lap}{\Delta}
\newcommand{\Grad}{\nabla}
\newcommand{\bdy}{\partial}

\newtheorem{thm}{Theorem}
\newtheorem{coro}{Corollary}
\newtheorem{prop}{Proposition}
\newtheorem{lem}{Lemma}

\newtheorem{remark}{Remark}

\numberwithin{equation}{section}
\numberwithin{lem}{section}
\numberwithin{prop}{section}
\numberwithin{claim}{section}

\title{A Harnack-Type Inequality for a Prescribing Curvature equation on a Domain with Boundary}
\subjclass{35J60, 35J55}
\author{
  Mathew R. Gluck 
  \and
  Ying Guo
  \and
  Lei Zhang
}
\date{}
\begin{document}

\address{Department of Mathematics\\
University of Florida\\
358 Little Hall, PO Box 118105\\
Gainesville FL 32611-8105}

\email{mgluck@ufl.edu}
\email{yguo@ufl.edu}
\email{leizhang@ufl.edu}
\maketitle
\begin{abstract}
	In this paper we use the method of moving spheres to derive a Harnack-type inequality for positive solutions of
	
	\begin{equation*}
		\left\{
			\begin{array}{ll}
				\lap u + K(x)u^{(n+2)/(n - 2)} = 0
				&
				x\in B_{1}^+\subset \bb R_+^n
				\\
				\frac{\partial u}{\partial x_n} = c(x) u^{n/(n - 2)}
				&
				x\in \bdy B_{1}^+ \intersect \bdy \bb R_+^n,
			\end{array}
		\right.
	\end{equation*}
	where $n\geq 4$,  $\bb R_+^n$ is the upper half-space and $B_1^+$ is the upper half unit ball. Under suitable assumptions on $K(x)$ and $c(x)$, we show that there is a positive constant $C$ such that for all positive solutions $u$, a Harnack type inequality holds. As a consequence of this inequality we obtain the following energy estimate
	
	\begin{equation*}
		\int_{B_{1/2}^+}\left( u^{\frac{2n}{n - 2}} + \abs{\Grad u}^2\right)\; dx \leq C.
	\end{equation*}
\end{abstract}

\section{Introduction}
	In conformal geometry the well known Yamabe problem asks if it is always possible to deform the metric of a compact Riemannian manifold to make the scalar curvature constant. The Yamabe problem can be translated to finding a solution to a semi-linear elliptic equation called the Yamabe equation. Through the works of Trudinger \cite{Trudinger1968},  Aubin \cite{Aubin1976} and Schoen \cite{Schoen1989} it is proved that the Yamabe equation always has a solution. A corresponding question is called Yamabe compactness problem, which asks if all solutions to the Yamabe equation are uniformly bounded when the manifold is not conformally diffeomorphic to the standard sphere. The Yamabe compactness problem was eventually proved to be affirmative if the dimension of the manifold is no greater than $24$ by Khuri-Marques-Schoen \cite{khuri}, and negative by Brendle-Marques \cite{brendle2} for dimensions greater than $24$. A central theme in these works and other works related to the Yamabe problem is the delicate analysis of solutions of the Yamabe equation that `blow up'. This analysis provides pointwise estimates for blow-up solutions and ultimately ensures that blowing up of solutions can only occur in certain ways. In the purely local setting, one avenue toward obtaining such estimates for blow-up solutions is to obtain a Harnack-type inequality.

If the manifold has a boundary, a natural question similar to the Yamabe problem is whether it is possible to deform the metric to change the scalar
curvature and the boundary mean curvature to specific functions (see Cherrier \cite{cherrier}). Suppose $(M^n,g)$ $(n\ge 3)$ is a Riemannian manifold with boundary $\bdy M$, let $\Hat g = u^{4/(n - 2)} g$ be a conformal to $g$, then the scalar curvature $R_g$ and boundary mean curvature $h_g$ of $g$ are related to the scalar curvature $K(x)$ and boundary mean curvature $c(x)$ of $\Hat g$ by the equations

\begin{equation}\label{eq:Curvature-Relations}
	\left\{
		\begin{array}{ll}
			K
			&
			=
			- \frac{4(n - 1)}{n - 2} u^{- \frac{n + 2}{n - 2}}
			\left(\lap_g u-\frac{n - 2}{4(n - 1)} R_g u\right)
			\\
			c
			&
			=
			\frac{2}{n - 2} u^{-\frac{n}{n - 2}} \left( \partial_{\nu_g} u
			+ \frac{n - 2}{2} h_g u\right),
		\end{array}
	\right.
\end{equation}
where $\lap_g$ is the Laplace-Beltrami operator with sign convention $- \lap_g\geq 0$ and $\nu_g$ is the unit outer normal vector on $\bdy M$.
If $K$ and $c$ are constants, finding a solution to (\ref{eq:Curvature-Relations}) is called the boundary Yamabe problem (BYP). Unlike its boundary-free counterpart, the BYP is not yet completely solved. Important progress has been made by Escobar \cite{Escobar1992Annals, Escobar1992DG}, Han-Li
\cite{HanLi1999,hanli2}, Marques \cite{Marques2005},etc.

Corresponding to the BYP, a compactness question can still be asked, which can be translated to asking whether there is a uniform bound for all the solutions satisfying (\ref{eq:Curvature-Relations}) under certainly assumptions. There is a vast literature on the uniform estimate of solutions to the BYP. The readers may look into \cite{demoura1,demoura2,djadli,felli1,HanLi1999,hanli2} and the references therein for extended discussion. To fully understand the BYP and the related compactness problem, it is crucial to understand the asymptotic behavior of blowup solutions near their blowup points. In this article we study the following locally defined equation:

\begin{equation}\label{eq:Main}
	\left\{
		\begin{array}{ll}
			\lap u + K(x) u^{\frac{n + 2}{n - 2}} = 0 & \text{ in } B_{1}^+ \subset \mathbb R^n_+, \quad u>0\\
			\frac{\partial u}{\partial x_n} = c(x) u^{\frac{n}{n - 2}}
				& \text{ on } \bdy B_{1}^+ \intersect \bdy \bb R_+^n,
		\end{array}
	\right.
\end{equation}

Our main goal in this article is to prove the following Harnack type inequality:
\begin{equation}\label{hartype}
(\max_{\overline{ B_{1/3}}} u )\cdot (\min_{\overline{ B_{2/3}}} u )\le C
\end{equation}
for some $C>0$.

Harnack inequality (\ref{hartype}) reveals important information on the interaction of bubbles. It implies that all bubbles have comparable magnitude and stay far away from one another. As a consequence, an energy estimate of the following type is essentially implied:
\begin{equation}\label{enerest}
		\int_{B_{1/2}^+} \left( \abs{\Grad u}^2 + u^{\frac{2n}{n - 2}}\right)\; dx
			\leq
			C.
	\end{equation}

To the best of our knowledge, Harnack type inequality similar to (\ref{hartype}) was first discovered for prescribing scalar curvature equations (with no boundary term) by Schoen \cite{sstan}, Schoen-Zhang \cite{schoen-zhang} and Chen-Lin \cite{ChenLin1997}. In 2003 the third author and Li \cite{LiZhang2003} proved (\ref{hartype}) for equation \ref{eq:Main} when $K$ and $c$ are both constants. In 2009 the third author proved (\ref{hartype}) for the case $n=3$ only assuming $K>0$ and $c$ to be smooth functions. In this article we derive (\ref{hartype}) and the energy estimate (\ref{enerest}) under natural assumptions on $K$ and $c$ for $n\ge 4$. It is evident from the previous work of the third author and Li \cite{lizhang-low,lizhang-c2,lizhang-c3}  that inequality (\ref{hartype}) is a crucial step toward obtaining fine estimates for solutions of \eqref{eq:Main}. Comparing with the results of Li-Zhang \cite{LiZhang2003} for $K,c$ being constants and Zhang \cite{Zhang2009} for $n=3$. The case of $n\ge 4$ with non constant coefficient functions is much harder. By constraining $K$ and $c$ appropriately, we are able to handle these new complications and derive the desired estimates. Specifically, we assume throughout this article that $n\geq 4$ and that $K$ satisfies

\begin{enumerate}[label=(K\arabic{*}), ref=(K\arabic{*}),leftmargin=5.0em]
	\item \label{item:K1}
		$K\in C^{n - 2}(\overline{B_{1}^+})$, and there exists a positive constant $C_0$ such that for all $x\in \overline{B_{1}^+}$,
	
		\begin{equation}\label{eq:CLassumption}
			\abs{ \Grad^j K(x) } \leq C_0\abs{\Grad K(x)}^{\frac{n-2 - j}{n-3}}
				\qquad
				j = 1,\cdots, n-2.
		\end{equation}
		\item \label{item:K2}
		There exists a constant $\Lambda>0$ such that both
	
		\begin{equation*}
			\frac 1 \Lambda \leq K(x)\; \text{ for all } x\in \overline{B_{1}^+}
			\qquad \text{ and } \qquad
			\norm{\Grad K(x)}_{C^{n - 2}(\overline{B_{1}^+})} \leq \Lambda.
		\end{equation*}
	\item \label{item:K3}
		$K$ depends only on $x_1, \cdots, x_{n - 1}$.
\end{enumerate}
There are many functions satisfying the assumptions on $K$. One elementary such function is

\begin{equation*}
	K(x)
		=
		1 + \left( \sum_{j = 1}^{n - 1} x_j^2\right)^{\alpha},\quad \alpha\ge \frac{n-2}2.
\end{equation*}
The flatness assumption \ref{item:K1} was used by Chen and Lin in \cite{ChenLin1997} to derive (among other results) a Harnack-type inequality for positive classical solutions of $\lap u + K(x) u^{(n+2)/(n - 2)}=0$ on $B_1$. Our approach is motivated by the approach taken by Chen and Lin. However, since the situation in this article involves $B_{1}^+$ instead of $B_{1}$, we must overcome complications that were not present in Chen and Lin's boundary-free case.
The main theorem of this article is the following.
\begin{thm}\label{thm:Main} Let $u$ be a solution of \eqref{eq:Main}.
	Suppose $K$ satisfies \ref{item:K1}, \ref{item:K2} and \ref{item:K3} and that $c$ is constant. There exist constants $C(n,\Lambda, C_0)>0$ and $\epsilon(n, \Lambda, C_0)>0$ such that if $c<\epsilon$, (\ref{hartype}) holds.
\end{thm}
\noindent In fact, Theorem \ref{thm:Main} holds under slightly less restrictive assumptions on $K$. Specifically, assumption \ref{item:K1} only needs to be satisfied in a neighborhood of the set of critical points of $K$. See for example \cite{ChenLin1997}. For simplicity, we allow $K$ to enjoy this property on all of $\overline{B_1^+}$.

As a corollary to Theorem \ref{thm:Main}, we have the following energy bound.

\begin{coro}\label{coro:Energy-Estimate}
	Suppose $u$, $K$, $c$ and $\epsilon$ are as in Theorem \ref{thm:Main}. There exists a positive constant $C(n, \Lambda,C_0)$ such that for all positive solutions $u$ of \eqref{eq:Main}, (\ref{enerest}) holds.
	\end{coro}
\noindent This energy estimate is a reflection of the fact that so called `bubbles', the large local maximum points of blow-up solutions to \eqref{eq:Main}, must stay far away from each other. \\

In view of the re-scaling $u(x) \mapsto R^{(n- 2)/2}u(Rx)$, Theorem \ref{thm:Main} implies corresponding Harnack inequalities on $B_R$ in general,
as long as the scalar curvature function and the mean curvature function still satisfy the same assumptions after scaling. The proof of Theorem \ref{thm:Main} is by contradiction. By the contradiction assumption, we obtain a sequence of blow-up solutions of \eqref{eq:Main}. After showing that blow-up can only occur near $\bdy B_1^+\intersect \bdy \bb R_+^n$, we use the method of moving spheres to derive a contradiction.

This paper is organized as follows. In Section \ref{section:Rescaling-and-Selection} we  use a standard selection process of Schoen \cite{Schoen1989} and Li \cite{Li1995} and the classification theorems of Caffarelli-Gidas-Spruck \cite{CaffarelliGidasSpruck1989} and Li-Zhu \cite{LiZhu1995} to obtain a convenient rescaling of the blow-up solutions. In Section \ref{section:Proof of Proposition} we show that blow-up points must be close to $\bdy B_1^+\intersect \bdy\bb R_+^n$, see Proposition \ref{prop:Ti-Unbounded}. This is achieved through three applications of the method of moving spheres (MMS). In particular MMS is first used to show that $\Grad K$ must vanish at a blow-up point, then MMS is used again to show that $\Grad K$ must vanish rapidly at a blow-up point, and a final application of MMS is used to show that blow-up can only occur near $\bdy B_1^+\intersect \bdy \bb R_+^n$. In Section \ref{section:Proof of Theorem}, we prove the Harnack-type inequality of Theorem \ref{thm:Main}. As in the proof that blow-up can only occur near $\bdy B_1^+\intersect \bdy \bb R_+^n$, the proof of Theorem \ref{thm:Main} is via three application of MMS; once to show $\Grad K$ vanishes at a blow-up point, once to show $\Grad K$ vanishes rapidly at a blow-up point, and finally to complete the proof the Theorem \ref{thm:Main}.  In Section \ref{section:Energy-Estimate} we give an overview of how to obtain the energy estimate in Corollary \ref{coro:Energy-Estimate} from the Harnack-type inequality. Since the derivation of Corollary \ref{coro:Energy-Estimate} from Theorem \ref{thm:Main} is standard, only the main points of the proof will be mentioned. The interested reader can consult, for example \cite{Li1995}, \cite{HanLi1999} and \cite{LiZhang2003} for details.

As notational conventions, we will use the following. The critical exponent $(n +2)/(n-2)$ will be denoted by $n^*$. We will use $o(1)$ to denote any quantity that tends to zero as $i\to\infty$. The symbols $C$, $C_1$ and $C_2$ will denote constants that depend only on $n$ and $\Lambda$ and will be different from line to line. The functions $v_{i,R}$ and $U_R$ as well as the domains $\Omega_i$ and $\Sigma_\lambda$ (to be defined) will be used in both Sections \ref{section:Proof of Proposition} and \ref{section:Proof of Theorem}, but will have different definitions in those sections.
	
\section{Rescaling and Selection}
	\label{section:Rescaling-and-Selection}
	
Suppose the Harnack-type inequality \eqref{hartype} fails. For each $i\in \bb N$, there is a positive solution $u_i$ of \eqref{eq:Main} with $K$ replaced with $K_i$ and $c$ replaced by $c_i$ such that
\begin{equation}\label{eq:Harnack-Type-Fail}
	\left(
		\max_{\overline{B_{1/3}^+}} u_i
	\right)
	\left(
		\min_{\overline{B_{2/3}^+}} u_i
	\right)
	>i.
\end{equation}
Note that $\Lambda$ and $C_0$ as given in the assumptions on $K$ are uniform in $i$. Without loss of generality we assume
$$\lim_{k\to \infty} K_i(x_i)=n(n-2). $$
By a standard selection process, see for example Schoen \cite{sstan} and Li \cite{Li1995} we may choose $x_i\in B_{1/2}^+\intersect \overline{\bb R_+^n}$ such that, for some $\sigma_i\to 0$,
$$ u_i(x_i)\ge \max_{B_{1/3}+}u_i, \quad u_i(x_i) \geq u_i(x) \quad \forall x\in B(x_i,\sigma_i)\cap \mathbb R^n_+, $$ and
$u_i(x_i)^{\frac 2{n-2}}\sigma_i\to \infty. $
For such $x_i$, \eqref{eq:Harnack-Type-Fail} yields

\begin{equation}\label{eq:Blow-Up}
	u_i(x_i) \min_{\overline{B_{2/3}^+}} u_i >i,
\end{equation}
which implies $u_i(x_i)\to \infty$. If $u_i$ are positive solutions of \eqref{eq:Main} and $x_i$ are local maximum points of $u_i$ for which \eqref{eq:Blow-Up} holds, $u_i$ is said to blow up, and a blow-up point is the limit of any convergent subsequence of $x_i$ for which \eqref{eq:Blow-Up} occurs. Setting

\begin{equation}\label{eq:Mi-Gammai-Ti-Notation}
	M_i = u_i(x_i),
	\qquad
	\Gamma_i = M_i^{\frac{2}{n-2}},
	\qquad
	T_i = x_{in}\Gamma_i
	\qquad
	\text{ and }
	\qquad
	E_i= B(- \Gamma_ix_i, 2\Gamma_i) \intersect\{ y_n >- T_i\},
\end{equation}
and applying standard arguments using the classification theorems of Caffarelli-Gidas-Spruck \cite{CaffarelliGidasSpruck1989} and Li-Zhu \cite{LiZhu1995} the functions

\begin{equation}\label{eq:viBar-Definition}
	\bar{v}_i(y)  = \frac{1}{M_i} u_i(x_i + \Gamma_i^{-1} y),
	\qquad
	y\in \overline{E_i}
\end{equation}
converge in $C^2$ over finite domains in the following  two cases.

\begin{enumerate}[label= Case \arabic{*}:, ref=Case \arabic{*},leftmargin=5.0em]
	\item \label{case:Ti-Unbounded}
		If there is a subsequence along which $T_i\to \infty$, then after passing to a further subsequence, we have $\bar v_i\to U$ in $C_{\rm loc}^2(\bb R^n)$, where 
		
		\begin{equation}\label{eq:Ti-Unbounded-Standard-Bubble}
			U(y)
				=
				\left( 1 + \abs y^2\right)^{- \frac{n -2}{2}}.
		\end{equation}
	\item  \label{case:Ti-Bounded}
		If $\{T_i\}$ is bounded then after passing to a subsequence we assume that $T_i$ converges. In this case, after passing to a further subsequence, $\bar v_i$ converges in $C^2$ over compact subsets of $\bb R^n \intersect \{ y_n\geq- \lim_i T_i\}$ to a classical solution $U$ of
		\begin{equation}\label{eq:Standard-Bubble-Equation-Ti-Bounded}
			\left\{
				\begin{array}{ll}
					\lap U + n(n-2) U^{n^*} = 0
					&
					y\in \bb R^n \intersect \{y_n > - \lim_i T_i\}
					\\
					\frac{\partial U}{\partial y_n} = \lim_i c_i U^{n/(n - 2)}
					&
					y\in \{ y_n = - \lim_i T_i\}
					\\
					U(0) = 1 = \max_{y\in \bb R^n \intersect \{y_n\geq - \lim_i T_i\}} U(y).
				\end{array}
			\right.
		\end{equation}
\end{enumerate}

Since the selection process and application of the classification theorems are standard, their applications are not presented here.Similar techniques have been used in  \cite{ChenLin1997}, \cite{LiZhang2003}, \cite{Zhang2002}, \cite{Zhang2009},etc. \\

The proof of Theorem \ref{thm:Main} is now split into two steps according to \ref{case:Ti-Unbounded} and \ref{case:Ti-Bounded}. In the first step we prove \ref{case:Ti-Unbounded} cannot occur, which shows that blow-up cannot occur far away from $\bdy B_1^+ \intersect \bdy \bb R_+^n$. In the second step, with the knowledge that blow-up can only occur near $\bdy B_1^+ \intersect \bdy \bb R_+^n$, we prove Theorem \ref{thm:Main}.

\section{Blow-up Can Only Occur Near $\bdy B_1^+\intersect \bdy \bb R_+^n$}
	\label{section:Proof of Proposition}

The proof of Theorem \ref{thm:Main} relies on delicate analysis of the behavior of $u_i$ near a blow-up point. As a first step, we prove the following theorem which says that blow-up can only near $\bdy B_1^+ \intersect \bdy \bb R_+^n$. In this theorem, we only require $c$ to be bounded.

\begin{thm}\label{prop:Ti-Unbounded}
	Suppose $\{u_i\}$ is  a sequence of positive solutions of \eqref{eq:Main} that satisfies (\ref{eq:Harnack-Type-Fail}) and that $\abs{c(x)} \leq C$ for all $x\in \bdy B_1^+\intersect \bdy \bb R_+^n$ and some $C>0$. There exists a constant $C_1>0$ independent of $i$ such that if $x_i$ is a local maximizer of $u_i$ for which \eqref{eq:Blow-Up} holds, then
	
	\begin{equation*}
			x_{in}\; u_i(x_i)^{\frac 2{n - 2}}\leq C_1,
	\end{equation*}
	where $x_{in}$ denotes the $n^{\rm th}$ coordinate of $x_i$.
\end{thm}
The proof of Theorem \ref{prop:Ti-Unbounded} is by contradiction. Specifically, MMS will be used three times; first, in Subsection \ref{subsection:Grad-Ki-Vanishing-Ti-Unbounded} to show that $\Grad K_i(x_i)$ vanishes, second in Subsection \ref{subsection:Fast-Grad-Ki-Vanishing-Ti-Unbounded} to show that $\Grad K_i(x_i)$ vanishes rapidly, and finally in Subsection \ref{subsection:Proof-of-Proposition} to complete the proof of Theorem \ref{prop:Ti-Unbounded}. The argument in this section is similar to that in \cite{ChenLin1997}. However Chen-Lin used a complicated moving plane method which involves two Kelvin transformations and a translation. We modify their approach by using a much simpler moving spheres to make the picture much easier to understand ( see \cite{Zhang2002}).

Let $M_i$, $\Gamma_i$ and $T_i$ be as in \eqref{eq:Mi-Gammai-Ti-Notation} and consider the functions

\begin{equation*}
	v_i(y) = \frac{1}{M_i} u_i(x_i + \Gamma_i^{-1} y),
	\qquad
	y\in \overline B(0, \frac18 \Gamma_i)\intersect \{ y_n \geq - T_i\}
\end{equation*}
(for the proof of Theorem \ref{prop:Ti-Unbounded}, $v_i$ is the same as $\bar v_i$ in \eqref{eq:viBar-Definition} and we omit the ``bar" in the notation). Observe that if $y\in \bdy B(0,\Gamma_i/8) \intersect \{y_n\geq - T_i\}$, then by \eqref{eq:Harnack-Type-Fail}

\begin{equation*}
	v_i(y)
		=
		\Gamma_i^{2-n} M_i u_i(x_i + \Gamma_i^{-1}y)
		\geq
		C(n) i \abs y^{2-n}.
\end{equation*}
In fact, we may choose $\epsilon_i\to 0$ slowly such that

\begin{equation}\label{eq:vi-Lower-Bound-Upper-Boundary}
	v_i(y)
		\geq
		\sqrt i \abs y^{2-n},
		\qquad
		y\in \bdy B(0,\epsilon_i\Gamma_i)\intersect\{y_n\geq - T_i\}.
\end{equation}
Define

\begin{equation*}
	\Omega_i = B(0,\epsilon_i\Gamma_i)\intersect\{y_n>- T_i\},
	\qquad
	\bdy'\Omega_i = \bdy\Omega_i\intersect\{y_n = - T_i\}
	\qquad
	\text{ and }
	\qquad
	\bdy''\Omega_i = \bdy\Omega_i \setminus \bdy'\Omega_i.
\end{equation*}
Elementary computations show that $v_i$ satisfies

\begin{equation}\label{eq:vi-Equations}
	\left\{
		\begin{array}{ll}
			\lap v_i + H_i(y) v_i^{n^*} = 0
			&
			y\in \Omega_i
			\\
			\frac{\partial v_i}{\partial y_n} = c_i(x_i + \Gamma_i^{-1} y) v_i^{n/(n - 2)}
			&
			y\in \bdy' \Omega_i,
		\end{array}
	\right.
\end{equation}
where $H_i(y) = K_i(x_i + \Gamma_i^{-1} y)$. By the contradiction hypothesis, there is a subsequence of $T_i$ along which $T_i\to \infty$, so \ref{case:Ti-Unbounded} applies. Before we can prove Theorem \ref{prop:Ti-Unbounded} we need to show that $\Grad K_i(x_i)$ vanishes rapidly. This will be done in two steps. The first step shows that $\Grad K_i(x_i)$ vanishes and is proven in Subsection \ref{subsection:Grad-Ki-Vanishing-Ti-Unbounded}. The second step shows that $\Grad K_i(x_i)$ vanishes rapidly and is proven in Subsection \ref{subsection:Fast-Grad-Ki-Vanishing-Ti-Unbounded}. For notational convenience, in subsections \ref{subsection:Grad-Ki-Vanishing-Ti-Unbounded}-\ref{subsection:Proof-of-Proposition} we will use $\abs{\Grad K_i(x_i)} = \delta_i$.
%
%
%
\subsection{Vanishing of $\Grad K_i(x_i)$}
\label{subsection:Grad-Ki-Vanishing-Ti-Unbounded}

\begin{prop}\label{lem:Vanishing-Ti-Unbounded}
	There exists a subsequence along which $\delta_i\to 0$.
\end{prop}
The proof of Proposition \ref{lem:Vanishing-Ti-Unbounded} is by contradiction. Namely, we suppose there is $\delta>0$ such that $\inf_i\delta_i \geq \delta >0$ and use the moving sphere method to derive a contradiction. By assumption \ref{item:K3} we may assume with no loss of generality that there is a subequence along which

\begin{equation*}
	\frac{\Grad K_i(x_i)}{\delta_i} \to e = (1, 0, \cdots, 0).
\end{equation*}
For $R\gg1$ fixed and to be determined, define the translations

\begin{equation*}
	v_{R,i}(y) = v_i(y - Re)
	\qquad
	\text{ and }
	\qquad
	U_R(y) = U(y - Re)
\end{equation*}
and the Kelvin inversions

\begin{equation*}
	v_{R,i}^\lambda(y)
		=
		\left(\frac{\lambda}{\abs y}\right)^{n - 2} v_{R,i}(y^\lambda)
	\qquad
	\text{ and }
	\qquad
	U_R^\lambda(y)
		=
		\left(\frac{\lambda}{\abs y}\right)^{n - 2}U_R(y^\lambda),
\end{equation*}
where $\lambda>0$ and $y^\lambda = \lambda^2 y/\abs y^2$. Clearly $v_{R,i}$, $U_R$ and their Kelvin inversions are well-defined in $\Sigma_\lambda = \Omega_i\setminus \overline B_\lambda$. For notational convenience, we set $\bdy'\Sigma_\lambda = \bdy\Sigma_\lambda \intersect \{y_n = -T_i\}$. Setting $\lambda^* = \sqrt{1 + R^2}$ and computing directly, it is easy to see that

\begin{equation}\label{eq:Critical-Position-Inequalities}
	\left\{
		\begin{array}{lll}
			(U_R - U_R^\lambda)(y) >0
			&
			y\in \bb R^n\setminus \overline B_\lambda
			&
			\text{ if } \lambda< \lambda^*
			\\
			(U_R - U_R^\lambda)(y) <0
			&
			y\in \bb R^n\setminus \overline B_\lambda
			&
			\text{ if } \lambda>\lambda^*.
		\end{array}
	\right.
\end{equation}
For $\lambda_0 = R$ and $\lambda_1 = R +2$, we have $\lambda^*\in [\lambda_0, \lambda_1]$, so we only consider $\lambda$ in this range. Define

\begin{equation*}
	w^\lambda(y) = v_{R,i}(y) - v_{R,i}^\lambda(y)
	\qquad
	y\in \Sigma_\lambda.
\end{equation*}
For convenience, we suppress the $i$-dependence in this notation. Elementary computations show that $w^\lambda$ satisfies

\begin{equation}\label{eq:wLambda-Equations-Ti-Unbounded}
	\left\{
		\begin{array}{ll}
			L_i w^\lambda(y) = Q_1^\lambda(y)
			&
			y\in \Sigma_\lambda
			\\
			B_i w^\lambda(y) = Q_2^\lambda
			&
			y\in \bdy'\Sigma_\lambda
			\\
			w^\lambda(y) = 0
			&
			y\in \bdy\Sigma_\lambda \intersect \bdy B_\lambda,
		\end{array}
	\right.
\end{equation}
where

\begin{equation}\label{eq:Differential-Operators=Definitions}
	\begin{array}{l}
		L_i = \lap + H_i(y - Re)\xi_1(y)
		\\
		B_i = \frac{\partial}{\partial y_n} - c_i(x_i + \Gamma_i^{-1}(y - Re)) \xi_2(y)
	\end{array}
\end{equation}
are the interior and boundary operators respectively,

\begin{equation}\label{eq:Mean-Value-Coefficient-1-Ti-Unbounded}
	\xi_1(y)
		= n^* \int_0^1
		\left( t v_{R,i}(y) + (1 - t) v_{R,i}^\lambda(y)\right) ^{\frac{4}{n - 2}}
		\; dt
\end{equation}
\begin{equation}\label{eq:Mean-Value-Coefficient-2-Ti-Unbounded}
	\xi_2(y)
		= \frac{n}{n - 2} \int_0^1
		\left( t v_{R,i}(y) + (1 - t) v_{R,i}^\lambda(y)\right) ^{\frac{2}{n - 2}}
		\; dt
\end{equation}
are obtained from the mean value theorem,

\begin{equation}\label{eq:Interior-Error-Function}
	Q_1^\lambda(y) = (H_i(y^\lambda - Re) - H_i(y - Re))(v_{R,i}(y)^\lambda)^{n^*}
\end{equation}
is an error term to be controlled by a test function and

\begin{eqnarray}\label{eq:Q2-Lambda}
	Q_2^\lambda(y)
		& = &
		\left(
			c_i(x_i + \Gamma_i^{-1}(y - Re)) - c_i(x_i + \Gamma_i^{-1}(y^\lambda - Re))
		\right)
		\left( v_{R,i}^\lambda(y)\right)^{n/(n - 2)}
		\notag\\
		& &
		-
		\frac{\lambda^{n - 2}}{\abs y^{n+2}} T_i
		\left(
				(n - 2) \abs y^2 v_{R,i}(y^\lambda)
				+
				2\lambda^2 \lb \Grad v_{R,i}(y^\lambda), y \rb
		\right).
\end{eqnarray}
We need to construct  a test function $h^\lambda$ such that both

\begin{equation}\label{eq:Perturbation-Test-Function}
	h^\lambda(y) = \circ(1) \abs y^{2-n}
	\qquad
	y\in \Sigma_\lambda
\end{equation}
and

\begin{equation}\label{eq:Desired-Differential-Inequalities}
	\left\{
		\begin{array}{ll}
			L_i(w^\lambda + h^\lambda)(y) \leq 0
			&
			y\in \Sigma_\lambda
			\\
			B_i(w^\lambda + h^\lambda)(y) \leq 0
			&
			y\in \bdy'\Sigma_\lambda \intersect \overline {\mathcal O}_\lambda,
		\end{array}
	\right.
\end{equation}
where

\begin{equation}\label{eq:OLambda-Definition}
	{\mathcal O}_\lambda
		=
		\{y\in \Sigma_\lambda : (v_{R,i} - v_{R,i}^\lambda)(y) \leq v_{R,i}^\lambda(y)\}.
\end{equation}
Such a test function is a perturbation of $w^\lambda$ that allows the maximum principle to be applied. For our purposes, the maximum principle only needs to apply on ${\mathcal O}_\lambda$ because $w^\lambda>0$ off of ${\mathcal O}_\lambda$. \\

We begin with some helpful estimates. Define

\begin{equation*}
	\Omega_\lambda
		=
		\{
			y\in \Sigma_\lambda \intersect B_{2\lambda} :
			y_1 > 2\abs{ (y_2, \cdots, y_n)}
		\}.
\end{equation*}
%
\begin{lem}\label{lem:Scalar-Curvature-Function-Estimates}
	There exist positive constants $C_1$ and $C_2$ independent of $i$ and $\lambda$ such that for $i$ sufficiently large,
	\begin{equation*}
		\left\{
			\begin{array}{ll}
				H_i(y^\lambda - Re) - H_i(y - Re)
					\leq
					- C_1 \Gamma_i^{-1} (\abs y - \lambda)
				&
				y\in \Omega_\lambda
				\\
				\abs{H_i(y^\lambda - Re) - H_i(y - Re) }
					\leq
					 C_2 \Gamma_i^{-1} (\abs y - \lambda)
				&
				y\in \Sigma_\lambda \setminus \Omega_\lambda.
			\end{array}
		\right.
	\end{equation*}
\end{lem}

\begin{remark} Unless mentioned otherwise, constants $C_1,C_2$ are independent of $i$ and $\lambda$.
\end{remark}

\begin{proof}
	The proof is elementary and follows from the definition of $\Omega_\lambda$, the fact that $K\in C^1(\overline B_3^+)$ and the assumption $0<\delta\leq \inf_i\delta_i$.
\end{proof}
We also have the following estimates for $v_{R,i}^\lambda$.
\begin{lem}\label{lem:vRi-Estimates}
	There exist positive constants $C_1$ and $C_2$ such that for $i$ sufficiently large,
	\begin{equation*}
		C_1 \abs y^{2-n} \leq v_{R,i}^\lambda(y) \leq C_2 \abs y^{2-n}
		\qquad
		y\in \Sigma_\lambda \setminus \Omega_\lambda
	\end{equation*}
	and
	
	\begin{equation*}
		C_1\left(\frac{\lambda}{\abs y}\right)^{n - 2}
			\left(
				\frac{1}{1 + \abs { y - \lambda e}^2}
			\right)^{\frac{n - 2}{2}}
		\leq
		v_{R,i}^\lambda(y)
		\leq
		2
		\qquad
		y\in \Omega_\lambda.
	\end{equation*}
\end{lem}
%
\begin{proof}
	The second estimate follows immediately from the convergence of $v_{R,i}$ to $U_R$, the properties of $U_R$ and the fact that $\abs{\lambda - R}\leq 2$. For the first estimate, it suffices to show that there exists a positive constant $C$ such that $C^{-1} \abs y^{2-n} \leq U_R^\lambda(y) \leq C\abs y^{2-n}$ for $y\in \Sigma_\lambda \setminus \Omega_\lambda$. Since $\abs{y^\lambda - Re}\leq C\lambda$, we have
	
	\begin{equation*}
		U_R^\lambda(y) \geq \frac 1 C \abs y^{2-n},
		\qquad
		y\in \Sigma_\lambda \setminus \Omega_\lambda.
	\end{equation*}
	On the other hand, after performing elementary computations we get
	
	\begin{equation*}
		\max\left\{
			\abs{y_1^\lambda - R}, \abs{(y_2, \cdots, y_n)}
			\right\}
			\geq
			C\lambda,
			\qquad
			y\in \Sigma_\lambda \setminus \Omega_\lambda,
	\end{equation*}
	so
	
	\begin{equation*}
		U_R^\lambda(y)
			\leq
			C\abs y^{2-n}.
	\end{equation*}
\end{proof}
Combining the results of Lemmas \ref{lem:Scalar-Curvature-Function-Estimates} and \ref{lem:vRi-Estimates} we obtain $\lambda$-independent positive constants $a_1$ and $a_2$ such that both

\begin{equation}\label{eq:Q1Lambda-Estimate-Ti-Unbounded}
	Q_1^\lambda(y)
		\leq - a_1\Gamma_i^{-1} (\abs y - \lambda)
			\left(\frac{1}{1 +\abs{y - \lambda e}^2}
			\right)^{(n + 2)/2}
			\qquad
			y\in \Omega_\lambda
\end{equation}
and

\begin{equation}\label{eq:Abs-Q1Lambda-Estimates-Ti-Unbounded}
	\abs{Q_1^\lambda(y)} \leq
		\left\{
			\begin{array}{ll}
				a_2\Gamma_i^{-1} (\abs y - \lambda)
					& y \in \overline \Omega_\lambda \\
				a_2\Gamma_i^{-1} (\abs y - \lambda) \abs y^{-2-n}
					& y\in \Sigma_\lambda \setminus
					\Omega_\lambda
			\end{array}
		\right.
\end{equation}
The following lemma gives estimates for the coefficient functions $\xi_1$ and $\xi_2$.

\begin{lem}\label{lem:Mean-Value-Coefficient-Estimates-Ti-Unbounded}
	There exist positive constants $C_1$ and $C_2$ such that for $i$ sufficiently large,
	
	\begin{equation*}
		\xi_1 (y)
			\leq
			C_2 \abs y^{-4}
			\qquad
			y\in (\Sigma_\lambda \intersect \mathcal O_\lambda) \setminus B_{4\lambda},
	\end{equation*}
	
	\begin{equation*}
		\xi_1(y)
			\geq
			C_1\abs y^{-4}
			\qquad
			y\in \Sigma_\lambda \setminus \Omega_\lambda,		
	\end{equation*}
	and
	
	\begin{equation*}
		\xi_2(y)
			\leq
			C_2\abs y^{-2}
			\qquad
			y\in \bdy'\Sigma_\lambda \intersect \overline{\mathcal O}_\lambda.		
	\end{equation*}
\end{lem}
\begin{proof}
	The proof follows immediately from the expressions of $\xi_1$ and $\xi_2$ in \eqref{eq:Mean-Value-Coefficient-1-Ti-Unbounded} and \eqref{eq:Mean-Value-Coefficient-2-Ti-Unbounded} and Lemma \ref{lem:vRi-Estimates}.
\end{proof}
The next lemma gives a useful estimate for $Q_2^\lambda$ and is the reason the proof of Theorem \ref{prop:Ti-Unbounded} is less difficult than the proof of Theorem \ref{thm:Main}.

\begin{lem}\label{lem:Q2Lambda-Estimate}
	There exists a constant $C>0$ such that for $i$ sufficiently large,
	
	\begin{equation*}
		Q_2^\lambda(y)
			\leq
			- CT_i \lambda^{n - 2} \abs y^{-n}
		\qquad
		y\in \bdy'\Sigma_\lambda.
	\end{equation*}
\end{lem}
\begin{proof}
	Since $\norm{c_i}_{L^\infty}\leq \Lambda$ and by Lemma \ref{lem:vRi-Estimates}, there is a positive constant $C$ such that
	
	\begin{equation*}
		(c_i(x_i + \Gamma_i^{-1}(y - Re)) - c_i(x_i + \Gamma_i^{-1}(y^\lambda - Re))
		(v_{R,i}^\lambda(y))^{\frac{2}{n - 2}}
		\leq
		C\abs y^{-2}
		\leq
		C T_i^{-2}.
	\end{equation*}
On the other hand, since $v_{R,i}\to U_R$ in $C^2(\overline B_{2\lambda})$ and since $\abs y\geq T_i$, if $i$ is sufficiently large,

\begin{eqnarray*}
	\abs y^2 v_{R,i}(y^\lambda)
	+ 2\lambda^2 \lb \Grad v_{R,i}(y^\lambda), y\rb
		&\geq &
		\frac 12 \abs y^2 \inf_{\overline B_\lambda} U_R(y)
		- 4\lambda^2 \norm{\Grad U_R}_{C^0(\overline B_\lambda)}\abs y
		\\
		&\geq &
		\frac 14 \abs y^2 \inf_{\overline B_\lambda} U_R(y).
\end{eqnarray*}
Lemma \ref{lem:Q2Lambda-Estimate} now follows from these two estimates and equation \eqref{eq:Q2-Lambda}.
\end{proof}
We now proceed with the construction of the test function $h^\lambda$. Let $\sigma_n$ denote the area of $\bb S^{n - 1}$ and let $G(y, \eta)$ be Green's function for $-\lap$ on $\bb R^n\setminus \overline B_\lambda$ relative to the Dirichlet condition. Recall that

\begin{equation}\label{eq:Greens-Function-Formula}
	G(y, \eta)
		=
		\frac{1}{(n - 2)\sigma_n}
		\left(
			\abs{ y - \eta}^{2-n}
			-
			\left(\frac{\abs y}{\lambda}\right)^{2-n}
			\abs{y^\lambda - \eta}^{2-n}
		\right).
\end{equation}
Estimates on $G$ are provided in Appendix \ref{subsection:Greens-Estimates-Appendix}. Define

\begin{equation}\label{eq:Naive-hLambda-Definition}
	h^\lambda(y)
		=
		\int_{\Sigma_\lambda} G(y, \eta) Q_1^\lambda(\eta) \; d\eta.
\end{equation}
By construction $h^\lambda$ satisfies

\begin{equation*}
	\left\{
		\begin{array}{ll}
			- \lap h^\lambda(y) = Q_1^\lambda(y)
			&
			y\in \Sigma_\lambda
			\\
			h^\lambda(y) = 0
			&
			y\in \bdy\Sigma_\lambda \intersect B_\lambda.
			\\
			\frac{\partial h^\lambda}{\partial y_n}(y)
				=
				\int_{\Sigma_\lambda}
				\frac{\partial G}{\partial y_n}(y, \eta)
				Q_1^\lambda(\eta)
				\; d\eta
			&
			y\in \bdy'\Sigma_\lambda.
		\end{array}
	\right.
\end{equation*}
We have the following estimates of $h^\lambda$.
\begin{lem}\label{lem:hLambda-Estimates-Ti-Unbounded}
	There exists $R_0$ sufficiently large such that if $R\geq R_0$ then there are positive constants $C_1$ and $C_2$ such that
	
	\begin{equation*}
		h^\lambda(y) \leq
			\left\{
				\begin{array}{ll}
					- C_1 \Gamma_i^{-1}(\abs y - \lambda)\lambda^{-n} \log \lambda
						& y\in \overline{\Sigma_\lambda}
						\intersect \overline B_{4\lambda}\\
					- C_1 \Gamma_i^{-1}\abs y^{2-n} \lambda^{-1}\log \lambda
						& y\in \overline{\Sigma_\lambda} \setminus B_{4\lambda}
				\end{array}
			\right.
	\end{equation*}
	
	and
	
	\begin{equation*}
		\abs{h^\lambda(y)} \leq
			\left\{
				\begin{array}{ll}
					C_2\Gamma_i^{-1}(\abs y - \lambda)\lambda^2
						& y\in \overline{\Sigma_\lambda}
						\intersect \overline B_{4\lambda}\\
					C_2\Gamma_i^{-1}\abs y^{2-n} \lambda^{n+1}
						& y\in \overline{\Sigma_\lambda} \setminus B_{4\lambda}.
				\end{array}
			\right.
	\end{equation*}
	\end{lem}
	%
\begin{proof}
	We consider separately the case $y\in \overline\Sigma_\lambda\intersect \overline B_{4\lambda}$ and the case $y\in \overline\Sigma_\lambda\setminus \overline B_{4\lambda}$. \\
	\emph{Case 1: $y\in \overline\Sigma_\lambda\intersect \overline B_{4\lambda}$}. Set
	
	\begin{equation*}
		I_1(y)
		=
		\int_{\Omega_\lambda} G(y, \eta) Q_1^\lambda(\eta)\; d\eta
		\quad \text{ and } \quad
		I_2(y)
		=
		\int_{\Sigma_\lambda \setminus\Omega_\lambda}
		G(y, \eta)  Q_1^\lambda(\eta)\; d\eta,
	\end{equation*}
	so $h^\lambda(y) = I_1(y) + I_2(y)$. By direct computation we have
	
	\begin{equation*}
		\int_{\Omega_\lambda}
			\frac{(\abs \eta - \lambda)^2}
			{
			(1 + \abs{ \eta - \lambda e}^2)^{(n+2)/2}
			}
			\; d\eta
		\geq
		C\log \lambda,
	\end{equation*}
	so using \eqref{eq:Q1Lambda-Estimate-Ti-Unbounded} the estimate of Green's function in \eqref{eq:GLowerBoundySmall}, the estimate for $I_1$ is
	
	\begin{eqnarray}\label{eq:I1EstimateySmall}
			I_1(y)
				& \leq &
				- C\Gamma_i^{-1} \int_{\Omega_\lambda} G(y, \eta)
					\frac{\abs \eta - \lambda}
						{(1 + \abs{\eta - \lambda e}^2)^{(n + 2)/2}}
					\; d\eta
				\notag\\
				& \leq&
				 - C \Gamma_i^{-1} (\abs y - \lambda)\lambda^{-n}
					\int_{\Omega_\lambda}
					\frac{(\abs \eta - \lambda)^2}
						{(1 + \abs{ \eta - \lambda e}^2)^{(n + 2)/2}}
					\; d\eta
				\notag\\
				& \leq&
				 - C \Gamma_i^{-1}(\abs y - \lambda)\lambda^{-n} \log \lambda.
	\end{eqnarray}
	To estimate $I_2$, let
	
	\begin{equation}\label{eq:I2-Estimate-Partition}
		\begin{array}{l}
			A_1 = \{\eta\in \Sigma_\lambda: \abs{ y - \eta}\leq (\abs y - \lambda)/3\},\\
			A_2 = \{\eta\in \Sigma_\lambda:
				\abs{ y - \eta}\geq (\abs y - \lambda)/3\text{ and }
				\abs \eta \leq 8\lambda\},\\
			A_3 = \{ \eta \in \Sigma_\lambda: \abs \eta \geq 8\lambda\},
		\end{array}
	\end{equation}
	and use \eqref{eq:Abs-Q1Lambda-Estimates-Ti-Unbounded} to write $I_2(y) \leq \sum_{k = 1}^3 I_2^{k}(y)$, where
	
	\begin{equation*}
		I_2^k(y)
			=
			\Gamma_i^{-1} \int_{A_k\setminus \Omega_\lambda}
			G(y, \eta) (\abs \eta - \lambda) \abs \eta^{-2-n}\; d\eta,
			\qquad
			k = 1,2,3.
	\end{equation*}
	Using Lemma \ref{lem:GreensEstimates} and performing routine integral estimates using $\abs \eta - \lambda \leq C\abs{ y - \eta}$ for $I_2^2(y)$ we obtain
	
	\begin{equation*}
		I_2^k(y) \leq C\Gamma_i^{-1} (\abs y - \lambda) \lambda^{-n}
		\qquad
		k = 1,2,3.
	\end{equation*}
	Combining this with the estimate for $I_1(y)$ given in \eqref{eq:I1EstimateySmall} and using $R\leq \lambda$ we see that if $R$ is sufficiently large then
	
	\begin{equation*}
		h^\lambda(y) \leq - C \Gamma_i^{-1}(\abs y - \lambda)\lambda^{-n} \log \lambda
		\qquad
		y\in \overline \Sigma_\lambda \intersect \overline B_{4\lambda}.
	\end{equation*}
	To estimate $\abs{h^\lambda(y) }$ for $y\in \overline \Sigma_\lambda \intersect \overline B_{4\lambda}$, observe that the only negative term above is $I_1(y)$, so we only need to estimate $\abs{I_1(y)}$. Using \eqref{eq:Abs-Q1Lambda-Estimates-Ti-Unbounded} and \eqref{eq:GUpperBoundySmall}, we have
	
	\begin{eqnarray*}
		\abs{ I_1(y)}
			& \leq &
				C\Gamma_i^{-1} \int_{\Omega_\lambda}
					G(y, \eta) (\abs \eta - \lambda)
					\; d\eta
			\\
			& \leq &
				C\Gamma_i^{-1}
					\left(
						\lambda\int_{A_1} \abs{y - \eta}^{2-n} \; d\eta
						+
						\int_{A_2} \frac{(\abs y - \lambda)(\abs \eta^2 - \lambda^2)}
						{\lambda\abs{ y - \eta}^n}
						(\abs \eta - \lambda)
						\; d\eta
					\right)
			\\
			& \leq &
			C\Gamma_i^{-1}(\abs y - \lambda)\lambda^2, 	
	\end{eqnarray*}
	where we have used $\abs \eta - \lambda \leq C\abs{y - \eta}$ for $\eta\in A_2$. This completes the proof of Lemma \ref{lem:hLambda-Estimates-Ti-Unbounded} in the case $y\in \overline \Sigma_\lambda \intersect \overline B_{4\lambda}$. \\
	\emph{Case 2: $y\in \overline \Sigma_\lambda \setminus B_{4\lambda}$}. Let $I_1$ and $I_2$ be as in Case 1 so that $h_1= I_1 + I_2$. Using \eqref{eq:GLowerBoundyLarge} and \eqref{eq:Q1Lambda-Estimate-Ti-Unbounded} we have
	\begin{eqnarray}\label{eq:I1EstimateyLarge}
		I_1(y)
			& \leq &
			- C\Gamma_i^{-1}
				\int_{\Omega_\lambda}
					\frac{\abs \eta - \lambda}{\lambda}\abs y^{2-n}
					\frac{\abs \eta - \lambda}
						{(1 + \abs{\eta - \lambda e}^2)^{(n + 2)/2}}
				\; d\eta
			\notag\\
			& \leq &
			- C \Gamma_i^{-1} \abs y^{2-n} \lambda^{-1} \log \lambda.
	\end{eqnarray}
	To estimate $I_2$ set
	\begin{equation}\label{eq:I2-Estimate-Partition-y-Large}
		\begin{array}{ll}
			D_1 & = \{ \eta \in\Sigma_\lambda: \abs \eta <\abs y/2\}\\
			D_2 & = \{\eta \in \Sigma_\lambda: \abs \eta> 2\abs y\}\\
			D_3 & = \{\eta \in \Sigma_\lambda: \abs{y - \eta} <\abs y/2\}\\
			D_4 & = \{\eta \in \Sigma_\lambda: \abs{y - \eta} \geq \abs y/2\
				\text{ and } \abs y/2\leq \abs \eta\leq 2\abs y\},
		\end{array}
	\end{equation}
	and use both \eqref{eq:Abs-Q1Lambda-Estimates-Ti-Unbounded} and \eqref{eq:GUpperBoundyLarge} to write $I_2(y)\leq C\sum_{k= 1}^4 I_2^k(y)$, where
	
	\begin{equation*}
		I_2^k(y)
			= \Gamma_i^{-1}
			\int_{D_k \setminus \Omega_\lambda}
				\abs{y - \eta}^{2-n}\abs \eta^{-1-n}
			\; d\eta,
			\qquad
			k = 1, \cdots, 4.
	\end{equation*}
	Performing elementary integral estimates we obtain
	
	\begin{equation*}
		I_2^k(y)
			\leq
			C\Gamma_i^{-1}\abs y^{2-n}\lambda^{-1}
		\qquad
			k = 1, \cdots, 4,
	\end{equation*}
	so in view of \eqref{eq:I1EstimateyLarge}, after choosing $R$ (and hence $\lambda$) large we get
	
	\begin{equation*}
		h^\lambda(y)
			 \leq
			 - C\Gamma_i^{-1} \abs y^{2-n} \lambda^{-1}\log \lambda
		\qquad
			y\in \overline \Sigma_\lambda\setminus B_{4\lambda}.
	\end{equation*}
	It remains to estimate $\abs{h^\lambda(y)}$ for $y\in \overline \Sigma_\lambda \setminus B_{4\lambda}$. The only negative term above is $\int_{\Omega_\lambda}G(y, \eta) Q_1^\lambda(\eta)\; d\eta$, so we only need to estimate this term. Using \eqref{eq:Abs-Q1Lambda-Estimates-Ti-Unbounded} and \eqref{eq:GUpperBoundyLarge} we have
	
	\begin{eqnarray*}
		\abs{\int_{\Omega_\lambda}G(y, \eta) Q_1^\lambda(\eta)\; d\eta}
			& \leq  &
			C\Gamma_i^{-1}
				\int_{\Omega_\lambda}\abs{ y - \eta}^{2-n}(\abs \eta - \lambda)\; d\eta
			\\
			& \leq &
			C\Gamma_i^{-1}\abs y^{2-n}\lambda^{n+1}. 		
	\end{eqnarray*}
	Lemma \ref{lem:hLambda-Estimates-Ti-Unbounded} is established.
\end{proof}
	%
%
We have the following estimate for the boundary derivative of $h^\lambda$.

\begin{lem}\label{lem: hLambda-Boundary-Derivative-Estimate}
	The test function $h^\lambda$ satisfies
	\begin{equation*}
		\frac{\partial h^\lambda}{\partial y_n}(y)
			=
			\circ(1)\abs y^{-n},
		\qquad
		y\in \bdy'\Sigma_\lambda.
	\end{equation*}
\end{lem}
%
\begin{proof}
	By direction computation we have
	
	\begin{equation*}
		\sigma_n \left. \frac{\partial G}{\partial y_n}(y, \eta) \right|_{y\in \bdy'\Sigma_\lambda}
			=
			\frac{\eta_n -y_n}{\abs{ y - \eta}^n}
			- \left(\frac{\lambda}{\abs y}\right)^n
			\abs{ y^\lambda - \eta}^{-n}
			\left(
				T_i\left(\frac{\abs \eta}{\lambda}\right)^2 + \eta_n
			\right).
	\end{equation*}
	Partition $\Sigma_\lambda$ as in \eqref{eq:I2-Estimate-Partition-y-Large}. Then use \eqref{eq:Abs-Q1Lambda-Estimates-Ti-Unbounded} and perform standard integral estimates using $y_n = -T_i$ and $\Sigma_\lambda \subset B(0, \epsilon_i\Gamma_i)$ to obtain
	
	\begin{equation*}
		\int_{D_k}
			\abs{\frac{ \partial G}{\partial y_n}(y, \eta)}
			\abs{Q_1^\lambda(\eta)}
			\; d\eta
		=
		\circ(1)\abs y^{-n},
		\qquad
		k = 1, \cdots, 4.
	\end{equation*}
\end{proof}
%
By construction of $h^\lambda$ and since $h^\lambda\leq0$ in $\Sigma_\lambda$ we have $L_i(w^\lambda + h^\lambda)(y) \leq 0$ for $y\in \Sigma_\lambda$. Moreover, by Lemmas \ref{lem:Mean-Value-Coefficient-Estimates-Ti-Unbounded} - \ref{lem: hLambda-Boundary-Derivative-Estimate} we obtain $B_i(w^\lambda + h^\lambda)(y) \leq 0$ for $y\in \bdy'\Sigma_\lambda \intersect \overline{\mathcal O}_\lambda$, so $h^\lambda$ satisfies \eqref{eq:Desired-Differential-Inequalities}. \\

The next step is to show that the moving sphere process can start.

\begin{lem}\label{lem:Moving-Spheres-Can-Start}
	If $i$ is sufficiently large, then
	
	\begin{equation}\label{eq:Moving-Spheres-Start-Inequality}
		w^{\lambda_0}(y) + h^{\lambda_0}(y) >0,
		\qquad
		y\in \Sigma_{\lambda_0}.
	\end{equation}
\end{lem}
%
\begin{proof}
	If $R_1\gg R$ is fixed and large, then by the convergence of $w^{\lambda_0}$ to $U_R - U_R^{\lambda_0}$, the properties of $U_R - U_R^{\lambda_0}$ and Lemma \ref{lem:hLambda-Estimates-Ti-Unbounded} we have
	\begin{equation*}
		(w^{\lambda_0} + h^{\lambda_0})(y) >0
		\qquad
		y\in \Sigma_{\lambda_0}\intersect \overline B_{R_1}.
	\end{equation*}
	We only need to show \eqref{eq:Moving-Spheres-Start-Inequality} for $y\in \Sigma_{\lambda_0} \setminus B_{R_1}$. By direct computation it is easy to see that there exists $\epsilon_0(\lambda_0)>0$ such that
	
	\begin{equation}\label{eq:URLambda-Upper-Bound-R1}
		U_R^{\lambda_0}(y)
			\leq
			(1 - 5\epsilon_0) \abs y^{2-n},
			\qquad
			\abs y \geq R_1.
	\end{equation}	
	Moreover, by choosing $R_1$ larger if necessary, we may simultaneously achieve
	
	\begin{equation}\label{eq:UR-Lower-Bound-R1}
		U_R(y)
			\geq
			\left( 1- \frac{\epsilon_0}{2}\right)
			\abs y^{2-n},
		\qquad
		\abs y = R_1.	
	\end{equation}
	As an immediate consequence of \eqref{eq:URLambda-Upper-Bound-R1} and the convergence of $v_{R,i}$ to $U_R$ we have
	
	\begin{equation*}
		v_{R,i}^{\lambda_0}(y)
			\leq
			(1 - 4\epsilon_0)
			\abs y^{2-n}
		\qquad
		y\in \Sigma_{\lambda_0} \setminus B_{R_1}.
	\end{equation*}
	Since $h^{\lambda_0}(y) = \circ(1) \abs y^{2-n}$ in $\Sigma_{\lambda_0}$ Lemma \ref{lem:Moving-Spheres-Can-Start} will be established once we show
	
	\begin{equation*}
		v_{R,i}(y)
			>
			(1 - \epsilon_0)
			\abs y^{2-n}
		\qquad
		y\in \Sigma_{\lambda_0}\setminus B_{R_1}.
	\end{equation*}
	This will be achieved via the maximum principle. By the convergence of $v_{R,i}$ to $U_R$, inequality \eqref{eq:UR-Lower-Bound-R1} and \eqref{eq:vi-Equations}, if $i$ is sufficiently large the function
	
	\begin{equation*}
		f_i(y) = v_{R,i}(y) - (1 - \epsilon_0)\abs y^{2-n}
	\end{equation*}
	is superharmonic in $\Sigma_{\lambda_0}\setminus B_{R_1}$ and positive on $\bdy B_{R_1}$. Moreover, by \eqref{eq:vi-Lower-Bound-Upper-Boundary},
	
	\begin{equation*}
		f_i(y)
			\geq
			C\sqrt i \abs y^{2-n},
		\qquad
		y\in \bdy \Sigma_{\lambda_0} \intersect\{ \abs y = \epsilon_i \Gamma_i\}.
	\end{equation*}
	By the maximum principle, if $f_i$ attains a nonpositive minimum value on $\overline \Sigma_{\lambda_0} \setminus B_{R_1}$, this value must be achieved on $\bdy'\Sigma_{\lambda_0}$. We show that this cannot happen. Accordingly, suppose $y_i^*\in \bdy'\Sigma_{\lambda_0}$ satisfies
	
	\begin{equation}\label{eq:Nonpositive-Boundary-Minimizer}
		\min_{y\in\overline\Sigma_{\lambda_0}\setminus B_{R_1}} f_i(y)
			=
			f_i(y_i^*)
			\leq
			0.
	\end{equation}
	Since $y_i^*$ is a minimizer, $\frac{\partial f_i}{\partial y_n}(y_i^*)\geq 0$. On the other hand, using \eqref{eq:vi-Equations}, \eqref{eq:Nonpositive-Boundary-Minimizer} and the assumption $T_i\to\infty$, if $i$ is sufficiently large then
	
	\begin{eqnarray*}
		\frac{\partial f_i}{\partial y_n}(y_i^*)
			& = &
			c_i(x_i + \Gamma_i^{-1}(y - Re)) v_{R,i}(y_i^*)^{\frac{n}{n - 2}}
			-
			C(\epsilon_0) T_i \abs{y_i^*}^{-n}
			\\
			&\leq &
			\left( \sup_i \norm{c_i}_{C^0(\overline{B_3^+})} - C T_i\right)
			\abs{y_i^*}^{-n}
			\\
			& < &
			0,
	\end{eqnarray*}
	a contradiction.
\end{proof}
%

With Lemma \ref{lem:Moving-Spheres-Can-Start} proven we can finally prove Proposition \ref{lem:Vanishing-Ti-Unbounded}.

\begin{proof}[Proof of Proposition \ref{lem:Vanishing-Ti-Unbounded}]
By Lemma \ref{lem:Moving-Spheres-Can-Start},

\begin{equation*}
	\bar \lambda
		=
		\sup
		\{
			\lambda \in [\lambda_0, \lambda_1]:
			(w^\mu + h^\mu)(y)\geq 0 \; \text{ in } \; \Sigma_{\mu}
			\quad
			\text{ for all } \lambda_0 \leq \mu \leq \lambda
		\}
\end{equation*}
is well defined. We will show that $\bar \lambda = \lambda_1>\lambda^*$ which, together with \eqref{eq:Critical-Position-Inequalities} and the estimate $h^\lambda(y) = \circ(1) \abs y^{2-n}$ for $\lambda_0 \leq \lambda \leq \lambda_1$ contradicts the convergence of $v_{R,i}$ to $U_R$. \\

Suppose $\bar \lambda <\lambda_1$. By continuity of $\lambda \mapsto w^\lambda + h^\lambda$ we have

\begin{equation*}
	(w^{\bar\lambda} + h^{\bar \lambda})(y)
		\geq
		0
	\qquad
	y\in \Sigma_{\bar \lambda}.
\end{equation*}
Moreover, $w^{\bar\lambda} + h^{\bar \lambda}$ satisfies

\begin{equation*}
	\left\{
		\begin{array}{ll}
			L_i(w^{\bar\lambda} + h^{\bar \lambda})(y) \leq 0
			&
			y\in \Sigma_{\bar \lambda}
			\\
			B_i(w^{\bar\lambda} + h^{\bar \lambda})(y)\leq 0
			&
			y\in \bdy' \Sigma_{\bar \lambda}\intersect \overline{\mathcal O}_{\bar\lambda}
			\\
			(w^{\bar\lambda} + h^{\bar \lambda})(y) = 0
			&
			y\in \bdy \Sigma_{\bar \lambda} \intersect B_{\bar \lambda}.
		\end{array}
	\right.
\end{equation*}
By \eqref{eq:vi-Lower-Bound-Upper-Boundary} and the estimate $\abs{h^{\bar\lambda}(y)} = \circ(1) \abs y^{2-n}$, we have

\begin{equation*}
	(w^{\bar\lambda} + h^{\bar \lambda})(y)
		>	
		0
	\qquad
	y\in \bdy \Sigma_{\bar \lambda}\intersect\{ \abs y = \epsilon_i \Gamma_i\}.
\end{equation*}
The strong maximum principle now ensures that $(w^{\bar\lambda} + h^{\bar \lambda})(y)>0$ for $y\in \Sigma_{\bar \lambda}$. By Hopf's lemma,

\begin{equation*}
	\frac{\partial}{\partial \nu}(w^{\bar\lambda} + h^{\bar \lambda})(y)
		>
		0
	\qquad
	y\in \bdy B_{\bar \lambda},
\end{equation*}
where $\nu$ is the outer unit normal vector on $\bdy B_{\bar \lambda}$ (pointing into $\Sigma_{\bar \lambda}$). Exploiting the continuity of $\lambda \mapsto w^{\lambda} + h^{ \lambda}$ once more we obtain a contradiction to the maximality of $\bar \lambda$.
Proposition \ref{lem:Vanishing-Ti-Unbounded} is established.
\end{proof}

%
%
\subsection{Rapid Vanishing of $\Grad K_i(x_i)$}
\label{subsection:Fast-Grad-Ki-Vanishing-Ti-Unbounded}
In this section we show that $\Grad K_i(x_i)$ vanishes quickly.

\begin{prop}\label{lem:Fast-Vanishing-Ti-Unbounded}
	There exists a constant $C>0$ such that for all $i$,
	
	\begin{equation*}
		\delta_i\Gamma_i^{n - 3} \leq C.
	\end{equation*}
\end{prop}
As in the proof of Proposition \ref{lem:Vanishing-Ti-Unbounded}, the proof of Proposition \ref{lem:Fast-Vanishing-Ti-Unbounded} is by contradiction. For ease of notation, set

\begin{equation*}
	\ell_i = \delta_i^{\frac 1{n - 3}} \Gamma_i
\end{equation*}
and pass to a subsequence for which $\ell_i\to \infty$. As in the proof of Proposition \ref{lem:Vanishing-Ti-Unbounded} we assume that $\delta_i^{-1}\Grad K_i(x_i) \to e$ and consider the functions $v_{R,i}$ and $U_R$ as well as their Kelvin inversions $v_{R,i}^\lambda$ and $U_R^\lambda$. With $w^\lambda$ as before, the equalities in \eqref{eq:wLambda-Equations-Ti-Unbounded} are still satisfied  and we seek to construct a test function $h^\lambda$ such that \eqref{eq:Perturbation-Test-Function} and \eqref{eq:Desired-Differential-Inequalities} hold. \\

Before constructing $h^\lambda$, we begin with some useful estimates. The following estimate is analogous to the estimate given in lemma \ref{lem:Scalar-Curvature-Function-Estimates}.

\begin{lem}\label{lem:Scalar-Curvature-Function-Fast-Estimates}
	There exist positive constants $C_1$ and $C_2$ such that for $i$ sufficiently large,
	
	\begin{equation*}
		\left\{
			\begin{array}{ll}
				H_i(y^\lambda - Re) - H_i(y - Re)
					\leq
					- C_1 \Gamma_i^{-1}\delta_i (\abs y - \lambda)
					&
					y\in \Omega_\lambda
					\\
				\abs{H_i(y^\lambda - Re) - H_i(y - Re) }
					\leq
					C_2\Gamma_i^{-1} \delta_i (\abs y - \lambda)
					\sum_{j = 0}^{n - 3} \ell_i^{-j} \abs y^j
					&
					y\in \Sigma_\lambda.
			\end{array}
		\right.
	\end{equation*}
\end{lem}
\begin{proof}
	The proof follows routinely from the assumptions on $K$ and Taylor's theorem.
\end{proof}
By Lemmas \ref{lem:Scalar-Curvature-Function-Fast-Estimates} and \ref{lem:vRi-Estimates} we obtain positive $\lambda$-independent constants $a_1$ and $a_2$ such that

\begin{equation*}
	Q_1^\lambda(y)
		\leq
		- a_1 \Gamma_i^{-1} \delta_i(\abs y - \lambda)
		(1 + \abs{y - \lambda e}^2)^{-\frac{n+2}{2}}
	\qquad
	y\in \Omega_\lambda
\end{equation*}
and

\begin{equation*}
	\abs{Q_1^\lambda(y)}
		\leq
		a_2 \Gamma_i^{-1}\delta_i(\abs y - \lambda)
		\sum_{j = 0}^{n - 3} \ell_i^{-j} \abs y^{j-2-n}
	\qquad
	y\in \Sigma_\lambda.
\end{equation*}

We are now ready to construct the test function $h^\lambda$. In this case, the construction of $h^\lambda$ is more delicate than in Subsection \ref{subsection:Grad-Ki-Vanishing-Ti-Unbounded}. Indeed, $h^\lambda$ as defined in \eqref{eq:Naive-hLambda-Definition} is not be guaranteed to be nonpositive. This creates extra terms in the interior equation for $w^\lambda + h^\lambda$ that must be controlled. To overcome this we use $Q_1^\lambda$ to construct a function $\Hat Q^\lambda$ and define $h^\lambda$ by integrating Green's function against $\Hat Q^\lambda$. The advantage of this definition is that $\Hat Q^\lambda$ will control both $Q_1^\lambda $ and the extra terms created by the possibility of $h^\lambda$ being positive. \\

To construct $\Hat Q^\lambda$, first define

\begin{equation*}
	\mathcal C_\lambda
		=
		\{
		y\in \Omega_\lambda \intersect B(0, 3\lambda/2):
		y_1>4\abs{(y_2, \cdots, y_n)}
		\}
\end{equation*}
and let $f^\lambda$ be any smooth function satisfying both

\begin{equation*}
	f^\lambda(y)
		=
		\left\{
			\begin{array}{ll}
				- \frac{a_1}{2} (1 + \abs{ y - \lambda e}^2)^{-(n+2)/2}
				&
				y\in \mathcal C_\lambda
				\\
				2 a_2 \sum_{j = 0}^{n - 3} \ell_i^{-j} \abs y^{j -2-n}
				&
				y\in \Sigma_\lambda \setminus \Omega_\lambda
			\end{array}
		\right.
\end{equation*}
and

\begin{equation*}
	-\frac 34 a_1(1 + \abs{ y - \lambda e}^2)^{-\frac{n + 2}{2}}
		\leq
		f^\lambda(y)
		\leq
		3a_2 \sum_{j = 0}^{n - 3} \ell_i^{-j} \abs y^{j - 2- n}
		\qquad
		y\in \overline \Omega_\lambda \setminus \mathcal C_\lambda.
\end{equation*}
Set

\begin{equation}\label{eq:QLambdaHat-Definition}
	\Hat Q^\lambda(y)
		=
		\Gamma_i^{-1}\delta_i (\abs y - \lambda) f^\lambda(y)
		\qquad
		y\in \Sigma_\lambda
\end{equation}
and observe that $\Hat Q^\lambda$ enjoys the estimates

\begin{equation}\label{eq:QLambdaHat-Upper-Bounds}
	\Hat Q^\lambda(y)
		\leq
		\left\{
			\begin{array}{ll}
				- \frac{a_1}{2} \Gamma_i^{-1} \delta_i
				(\abs y - \lambda)
				(1 + \abs{ y - \lambda e}^2)^{-(n+2)/2}
				&
				y\in \overline{\mathcal C}_\lambda
				\\
				3 a_2 \Gamma_i^{-1}\delta_i
				(\abs y - \lambda)
				\sum_{j = 0}^{n - 3} \ell_i^{-j} \abs y^{j - 2 -n}
				&
				y\in \overline \Sigma_\lambda \setminus \mathcal C_\lambda
			\end{array}
		\right.
\end{equation}
and

\begin{equation}\label{eq:Abs-QLambdaHat-Estimate}
	\abs{\Hat Q^\lambda(y)}
		\leq
			\left\{
				\begin{array}{ll}
					C\Gamma_i^{-1}\delta_i (\abs y- \lambda)
					&
					y\in \Sigma_\lambda \intersect B_{4\lambda}
					\\
					C\Gamma_i^{-1} \delta_i (\abs y - \lambda)
					\sum_{j = 0}^{n - 3} \ell_i^{-j} \abs y^{j -2-n}
					&
					y\in \Sigma_\lambda \setminus B_{4\lambda}.
				\end{array}
			\right.
\end{equation}
Moreover, we have

\begin{equation}\label{eq:QLambda-QLambdaHat}
	(Q_1^\lambda - \Hat Q^\lambda)(y)
		\leq
		\left\{
			\begin{array}{ll}
				- \frac{a_1}{4}\Gamma_i^{-1} \delta_i
				(\abs y - \lambda)
				(1 + \abs{y - \lambda e}^2)^{-(n+2)/2}
				&
				y\in \overline \Omega_\lambda
				\\
				- a_2 \Gamma_i^{-1}\delta_i
				(\abs y - \lambda)
				\sum_{j = 0}^{n - 3} \ell_i^{-j} \abs y^{j - 2- n}
				&
				y\in \overline \Sigma_\lambda \setminus \Omega_\lambda.
			\end{array}
		\right.
\end{equation}
Define

\begin{equation*}
	h^\lambda(y)
		=
		\int_{\Sigma_\lambda} G(y, \eta) \Hat Q^\lambda(\eta)\; d\eta.
\end{equation*}
By construction, $h^\lambda$ satisfies

\begin{equation}\label{eq:Fast-hLambda-Equations}
	\left\{
		\begin{array}{ll}
			- \lap h^\lambda(y) = \Hat Q^\lambda(y)
			&
			y\in \Sigma_\lambda
			\\
			h^\lambda(y) = 0
			&
			y\in \bdy B_\lambda
			\\
			\frac{\partial h^\lambda}{\partial y_n}(y)
				=
				\int_{\Sigma_\lambda}
					\frac{\partial G}{\partial y_n}(y, \eta)
					\Hat Q^\lambda (\eta)
				\; d\eta
			&
			y\in \bdy'\Sigma_\lambda.
		\end{array}
	\right.
\end{equation}
The next lemma provides useful estimates for $h^\lambda$.

\begin{lem}\label{lem:Fast-hLambda-Estimates}
	If $R$ and $i$ are sufficiently large, then there are positive constants $C_1$ and $C_2$ such that both
	\begin{equation*}
		h^\lambda(y)
			\leq
			\left\{
				\begin{array}{ll}
					- C_1 \Gamma_i^{-1} \delta_i
					(\abs y - \lambda) \lambda^{-n} \log \lambda
					&
					y\in \overline B_{4\lambda}
					\\
					C_2 \Gamma_i^{-1} \delta_i \abs y^{2-n}
					\left(
						\ell_i^{-1} \log\frac{ \abs y }{\lambda}
						+ \sum_{j = 2}^{n - 3} \ell_i^{-j} \abs y^{j - 1}
					\right)
					&
					y\in \overline \Sigma_\lambda \setminus B_{4\lambda}
				\end{array}
			\right.
	\end{equation*}
	and
	
	\begin{eqnarray}\label{eq:Fast-Abs-hLambda-Estimate-Ti-Unbounded}
		\abs{h^\lambda(y) }
			& \leq &
			\left\{
				\begin{array}{ll}
					C_2\Gamma_i^{-1} \delta_i\lambda^2 (\abs y - \lambda)
					&
					y\in \overline \Sigma_\lambda \intersect \overline B_{4\lambda}
					\\
					C_2\Gamma_i^{-1} \delta_i \abs y^{2-n}
					\left(
						\lambda^{1+n}  + \ell_i^{-1}\log \frac{\abs y}{\lambda}
						+
						\sum_{j = 2}^{n -3} \ell_i^{-j} \abs y^{j -1}
					\right)
					&
					y\in \overline \Sigma_\lambda \setminus B_{4\lambda}
				\end{array}
			\right.
			\notag
			\\
			& = &  \circ(1) \abs y^{2-n}.
	\end{eqnarray}
\end{lem}
%
	\begin{proof}
	Write
	\begin{equation*}
		h^\lambda(y) = I_1(y) + I_2(y),
	\end{equation*}
	with
	\begin{equation}\label{eq:Fast-I1-I2-Definitions-Ti-Unbounded}
		I_1(y)
		=
		\int_{\mathcal C_\lambda}G(y, \eta) \widehat Q^\lambda(\eta) \; d\eta
		\qquad
		\text{ and }
		\qquad
		I_2(y)
		=
		\int_{\Sigma_\lambda \setminus \mathcal C_\lambda}
		G(y, \eta) \widehat Q^\lambda(\eta) \; d\eta.
	\end{equation}
	We consider separately the case $y\in \overline \Sigma_\lambda\intersect\overline B_{4\lambda}$ and the case $y\in \overline \Sigma_\lambda\setminus B_{4\lambda}$. \\
	\emph{Case 1:} $y\in \overline \Sigma_\lambda\intersect\overline B_{4\lambda}$.\\
	In this case, using the estimates for $\widehat Q^\lambda$ in \eqref{eq:QLambdaHat-Upper-Bounds} and the estimates for $G$ in \eqref{eq:GLowerBoundySmall} estimating similarly to \eqref{eq:I1EstimateySmall}  we obtain
	
	\begin{eqnarray}\label{eq:Fast-I1-Estimate-ySmall}
		I_1(y)
			& \leq & - C \Gamma_i^{-1} \delta_i (\abs y - \lambda)\lambda^{-n} \log\lambda.
	\end{eqnarray}
	To estimate $I_2(y)$, let $A_1, A_2$ and $A_3$ be as in \eqref{eq:I2-Estimate-Partition} and write $I_2(y) = \sum_{k = 1}^3 I_2^{k}(y)$,  where
	
	\begin{equation*}
		I_2^k(y)
			=
			\int_{A_k \setminus \mathcal C_\lambda}
			G(y, \eta) \widehat Q^\lambda(\eta) \; d\eta
			\\
	\end{equation*}
	Performing routine integral estimates using $\ell_i\to\infty$ and Lemma \ref{lem:GreensEstimates} yields
	
	\begin{equation*}
		\abs{I_2^k(y)}
			\leq
			C\Gamma_i^{-1} \delta_i(\abs y - \lambda) \lambda^{-n}
		\qquad
		k = 1,2,3.
	\end{equation*}
	Combining this with the estimate for $I_1(y)$ given in \eqref{eq:Fast-I1-Estimate-ySmall} and choosing $R$ sufficiently large we obtain
	
	\begin{equation*}
		h^\lambda(y)
			\leq
			- C \Gamma_i^{-1} \delta_i(\abs y - \lambda) \lambda^{-1} \log \lambda
			\qquad
			y\in \overline\Sigma_\lambda \intersect \overline B_{4\lambda}.
	\end{equation*}
	To show \eqref{eq:Fast-Abs-hLambda-Estimate-Ti-Unbounded} for $y\in \overline \Sigma_\lambda \intersect \overline B_{4\lambda}$ we only need to estimate $\abs{I_1(y)}$. Using \eqref{eq:Abs-QLambdaHat-Estimate} and the estimates for $G(y, \eta)$ in Lemma \ref{lem:GreensEstimates}, we have
		
	\begin{eqnarray}\label{eq:Absh1EstimateFastSmally}
		\abs{ I_1(y)}
			& \leq &
			\int_{\mathcal C_\lambda}G(y, \eta)
			\abs{ \widehat Q^\lambda(\eta)}\; d\eta
			\notag
			\\
			&\leq &
			C\Gamma_i^{-1}\delta_i
			\left(
			\lambda \int_{A_1}\abs{ y - \eta}^{2-n}\; d\eta
			+ \int_{A_2}
			\frac{(\abs y - \lambda)(\abs \eta^2 - \lambda^2)}{\lambda\abs{ y - \eta}^n}
			(\abs \eta - \lambda)\; d\eta
			\right)\notag
			\\
			& \leq &
			C\Gamma_i^{-1} \delta_i(\abs y - \lambda) \lambda^2.
	\end{eqnarray}
	\emph{Case 2:} $y\in \overline \Sigma_\lambda \setminus B_{4\lambda}$. \\
	By \eqref{eq:QLambdaHat-Upper-Bounds} and \eqref{eq:GLowerBoundyLarge} we have
	
	\begin{eqnarray*}
		I_1(y)
			& \leq &
			- C \Gamma_i^{-1} \delta_i\abs y ^{2-n} \lambda^{-1}\log \lambda.
	\end{eqnarray*}
	To estimate $I_2(y)$, let $D_1, D_2, D_3$ and $D_4$ be as in \eqref{eq:I2-Estimate-Partition-y-Large} and let $I_2^{k}(y) = \int_{D_k \setminus \mathcal C_\lambda} G(y, \eta)\widehat Q^\lambda (\eta)\; d\eta$ so that $I_2(y) = \sum_{k = 1}^4 I_2^k(y)$. For each $k = 1, \cdots, 4$ we use both \eqref{eq:Abs-QLambdaHat-Estimate} and \eqref{eq:GUpperBoundyLarge} to estimate $I_2^k(y)$. For $k = 1$ we have
	
	\begin{eqnarray*}
		\abs{I_2^{1}(y)}
			& \leq &
			C\Gamma_i^{-1} \delta_i \int_{D_1} G(y, \eta)
			\sum_{j = 0}^{n -3} \ell_i^{-j}\abs \eta^{j -1-n}\; d\eta
			\\
			& \leq &
			C\Gamma_i^{-1} \delta_i \abs y^{2-n}
			 \left(
			 \lambda^{-1} + \ell_i^{-1}\log\frac{\abs y}{\lambda}
			 +  \sum_{j = 2}^{n - 3} \ell_i^{-j}\abs y^{j-1}
			 \right).\
	\end{eqnarray*}
	For $k = 2,3,4$, the integrals $I_2^k$ are minor. After performing routine integral estimates we have
	
	\begin{equation*}
		\abs{I_2^k(y)}
			\leq
			C\Gamma_i^{-1} \delta_i
			\abs y^{2-n} \sum_{j = 0}^{n - 3} \ell_i^{-j} \abs y^{j - 1}.
			\qquad
			k = 2,3,4.
	\end{equation*}
	Combining the estimates for $I_2^{k}(y)$ , $k = 1, \cdots, 4$, we get
	
	\begin{equation}\label{eq:Fast-Abs-I2-Estimate-Ti-Unbounded}
		\abs{I_2(y) }
			\leq
			C\Gamma_i^{-1}\delta_i\abs y^{2-n}
			\left(
				\lambda^{-1}+ \ell_i^{-1}\log \abs y
				+ \sum_{j = 2}^{n - 3} \ell_i^{- j} \abs y^{j - 1}
			\right)
			=
			\circ(1) \Gamma_i^{-1} \abs y^{2-n}.
	\end{equation}
	Combining the estimates for $I_1$ and $I_2$ we obtain a positive constant $C$ such that for $R$ sufficiently large
	
	\begin{equation}\label{eq:Fast-hLambda-Upper-Bound-y-Large}
		h^\lambda(y)
			\leq
			C \Gamma_i^{-1}\delta_i \abs y^{2-n}
			\left(
				\ell_i^{-1}\log \abs y
				+ \sum_{j = 2}^{n -3} \ell_i^{-j}\abs y^{j -1}
			\right).
	\end{equation}
	Notice in particular that $h^\lambda(y)$ need not be negative.\\
	
	To show \eqref{eq:Fast-Abs-hLambda-Estimate-Ti-Unbounded}, by \eqref{eq:Fast-Abs-I2-Estimate-Ti-Unbounded}, we only need to estimate $\abs{ I_1(y)}$. By \eqref{eq:Abs-QLambdaHat-Estimate} and since $G(y, \eta)\leq C\abs{y - \eta}^{2-n}$ in $\mathcal C_\lambda$ we have
	
	\begin{eqnarray*}
		\abs{I_1(y)}
			& \leq&
			\int_{\mathcal C_\lambda} G(y, \eta) \abs{\widehat Q^\lambda(\eta)}\; d\eta
			\\
			&\leq &
			C\Gamma_i^{-1} \delta_i\abs y^{2-n} \lambda^{1 + n}.
	\end{eqnarray*}
	Lemma \ref{lem:Fast-hLambda-Estimates} is established.
\end{proof}
\begin{lem}\label{lem:Fast-hLambda-Boundary-Derivative-Estimate}
	The test function satisfies the estimate
	
	\begin{equation*}
		\frac{\partial h^\lambda}{\partial y_n}(y)=
			\circ(1) \abs y^{-n}
			\qquad
			y\in \bdy'\Sigma_\lambda.
	\end{equation*}
\end{lem}
\begin{proof}
	Use \eqref{eq:Abs-QLambdaHat-Estimate}, $\delta_i = \circ(1)$ and $\abs y\leq \epsilon_i\Gamma_i$ to obtain $\abs{ \Hat Q^\lambda(y) } = \circ(1) \Gamma_i^{-1} \abs y^{-1-n}$ for $y\in \Sigma_\lambda$. Now proceed as in the proof of Lemma \ref{lem: hLambda-Boundary-Derivative-Estimate}.
\end{proof}
\begin{proof}[Proof of Proposition \ref{lem:Fast-Vanishing-Ti-Unbounded}]
By \eqref{eq:Fast-hLambda-Equations}, \eqref{eq:QLambda-QLambdaHat} and Lemmas \ref{lem:Mean-Value-Coefficient-Estimates-Ti-Unbounded} and \ref{lem:Fast-hLambda-Estimates}, we have after increasing $R$ if necessary and for $i$ large

\begin{equation*}
	L_i(w^\lambda + h^\lambda)(y)
		=
		(Q_1^\lambda - \Hat Q^\lambda)(y)
		+ H_i(y - Re)\xi_1(y) h^\lambda(y)
		\leq
		0
		\qquad
		y\in \Sigma_\lambda\intersect \mathcal O_\lambda.
\end{equation*}
Moreover, by lemmas \ref{lem:Mean-Value-Coefficient-Estimates-Ti-Unbounded}, \ref{lem:Q2Lambda-Estimate}, \ref{lem:Fast-hLambda-Estimates}, and \ref{lem:Fast-hLambda-Boundary-Derivative-Estimate} we obtain

\begin{equation*}
	B_i(w^\lambda + h^\lambda)
		\leq
		0
		\qquad
		y\in \bdy'\Sigma_\lambda.
\end{equation*}
Arguing similarly to the proof of Lemma \ref{lem:Moving-Spheres-Can-Start} we see that the moving sphere process can start at $\lambda = \lambda_0$, then arguing similarly to the proof of Proposition \ref{lem:Vanishing-Ti-Unbounded} we obtain a contradiction to $\ell_i\to\infty$. Proposition \ref{lem:Fast-Vanishing-Ti-Unbounded} is established.
\end{proof}

%
%
\subsection{Completion of the Proof of Theorem \ref{prop:Ti-Unbounded}}
\label{subsection:Proof-of-Proposition}
With a rapid vanishing rate for $\delta_i$ in hand, we are ready to prove Theorem \ref{prop:Ti-Unbounded}.

\begin{proof}[Proof of Theorem \ref{prop:Ti-Unbounded}]

In this proof we consider the functions $v_i$, $U$ (not shifted by $Re$) as well as their Kelvin inversions

\begin{equation*}
	v_i^\lambda(y)
		=
		\left(\frac{\lambda}{\abs y}\right)^{n - 2}
		v_i(y^\lambda)
	\qquad
	\text{ and }
	\qquad
	U^\lambda(y)
		=
		\left(\frac{\lambda}{\abs y}\right)^{n - 2}
		U(y^\lambda)
\end{equation*}
for $y\in \Sigma_\lambda = \Omega_i\setminus \overline B_\lambda$. In this case, with $\lambda^* = 1$ direct computation yields

\begin{equation*}
	\left\{
		\begin{array}{lll}
			(U- U^\lambda)(y) >0
			&
			y\in \bb R^n \setminus \overline B_\lambda
			&
			\text{ if } \lambda <\lambda^*
			\\
			(U- U^\lambda)(y) <0
			&
			y\in \bb R^n \setminus \overline B_\lambda
			&
			\text{ if } \lambda >\lambda^*,
		\end{array}
	\right.
\end{equation*}
and we consider $\lambda$ between $\lambda_0 = 1/2$ and $\lambda_1 = 2$. Set

\begin{equation*}
	w^\lambda(y)
		=
		v_i(y) - v_i^\lambda(y)
		\qquad
		y\in \Sigma_\lambda.
\end{equation*}
Then $w^\lambda$ satisfies equations \eqref{eq:wLambda-Equations-Ti-Unbounded} - \eqref{eq:Interior-Error-Function} with $R= 0$. We still need to construct a test function $h^\lambda$ such that \eqref{eq:Perturbation-Test-Function} and \eqref{eq:Desired-Differential-Inequalities} hold. Because of the rapid vanishing rate of $\delta_i$, the construction will be simple. \\

By an application of Taylor's Theorem, assumption \ref{item:K1} and Proposition \ref{lem:Fast-Vanishing-Ti-Unbounded}, we have
	
	\begin{equation}\label{eq:HiLambda-Hi-Estimate-R=0}
		\abs{H_i(y^\lambda) - H_i(y)}
			\leq
			C\Gamma_i^{2-n} \abs y^{n - 2}
			\qquad
			y\in \Sigma_\lambda.
	\end{equation}
Since $\lambda\leq 2$, using the convergence of $v_i^\lambda$ to $U^\lambda$ and the properties of $U^\lambda$, we have

\begin{equation}\label{eq:viLambda-Upper-Bound}
	v_i^\lambda(y)
		\leq
		C\abs y^{2-n}
		\qquad
		y\in \Sigma_\lambda.
\end{equation}
Using this and \eqref{eq:HiLambda-Hi-Estimate-R=0} in the expression of $Q_1^\lambda$ we have

\begin{equation}\label{eq:Final-Q1Lambda-Estimate}
	Q_1^\lambda(y)
		\leq
		C\Gamma_i^{2-n} \abs y^{-4}
		\qquad
		y\in \Sigma_\lambda.
\end{equation}
Moreover, as in Lemma \ref{lem:Q2Lambda-Estimate} with $R = 0$ and $1/2\leq \lambda \leq 2$, we have

\begin{equation}\label{eq:Final-Q2Lambda-Estimate}
	Q_2^\lambda(y)
		\leq
		- CT_i \abs y^{-n}
		\qquad
		y\in \bdy'\Sigma_\lambda.
\end{equation}
Set

\begin{equation*}
	h^\lambda(y)
		=
		-a \Gamma_i^{2-n}(\lambda^{-1} - \abs y^{-1})
		\leq
		0,
		\qquad
		y\in \overline \Sigma_\lambda,
\end{equation*}
where $a>0$ is to be determined. By direct computation and since $\Sigma_\lambda \subset B(0, \epsilon_i\Gamma_i)$, $h^\lambda$ is seen to satisfy

\begin{equation}\label{eq:Final-hLambda-Equations}
	\left\{
		\begin{array}{ll}
			\lap h^\lambda(y)
				\leq - a \Gamma_i^{2-n}\abs y^{-3}
				&
				y\in \Sigma_\lambda
				\\
			h^\lambda(y)
				=
				0
				&
				y\in \bdy B_\lambda
				\\
			h^\lambda(y) = \circ(1) \abs y^{2-n}
				&
				y\in \Sigma_\lambda
				\\
			\frac{\partial h^\lambda}{\partial y_n}(y)
				= a T_i \Gamma_i^{2-n}\abs y^{-3}
				= \circ(1)\abs y^{-n}
				&
				y\in \bdy'\Sigma_\lambda.
		\end{array}
	\right.
\end{equation}
Combining \eqref{eq:Final-Q1Lambda-Estimate} and \eqref{eq:Final-hLambda-Equations}, after choosing $a$ sufficiently large we obtain

\begin{equation*}
	L_i(w^\lambda + h^\lambda)(y)
		\leq
		0
		\qquad
		y\in \Sigma_\lambda.
\end{equation*}
Moreover, by \eqref{eq:Final-Q2Lambda-Estimate}, Lemma \ref{lem:Mean-Value-Coefficient-Estimates-Ti-Unbounded} and \eqref{eq:Final-hLambda-Equations} we have

\begin{equation*}
	B_i(w^\lambda + h^\lambda)(y)
		\leq
		0,
		\qquad
		y\in \bdy'\Sigma_\lambda \intersect \mathcal O_\lambda,
\end{equation*}
where in this case $\mathcal O_\lambda$ is as in \eqref{eq:OLambda-Definition} with $R = 0$. Arguing similarly to the proof of Lemma \ref{lem:Moving-Spheres-Can-Start} shows that the moving sphere process can start at $\lambda_0 = 1/2$,then arguing as in the proof of Proposition \ref{lem:Vanishing-Ti-Unbounded} yields a contradiction. Theorem \ref{prop:Ti-Unbounded} is established.
\end{proof}

\section{Proof of Theorem \ref{thm:Main}}
	\label{section:Proof of Theorem}

In this section we prove the Harnack-type inequality. The proof is similar to the proof of Theorem \ref{prop:Ti-Unbounded} in that three application of MMS will be applied; first in Subsection \ref{subsection:Vanishing-Ti-Bounded} to show that $\Grad K$ vanishes at a blow-up point, second in Subsection \ref{subsection:Fast-Vanishing-Ti-Bounded} to show that $\Grad K$ vanished rapidly at a blow-up point, and finally in Subsection \ref{subsection:Completion-of-Proof-of-Theorem} to complete the proof of Theorem \ref{thm:Main}.  The essential difference between the proof of Theorem \ref{thm:Main} and the proof of Theorem \ref{prop:Ti-Unbounded} is that in the proof of Theorem \ref{thm:Main}, due to Theorem \ref{prop:Ti-Unbounded}, the complications presented by the boundary equations are not minor. This makes the construction of the test functions much more delicate in the proof of Theorem \ref{thm:Main} than in the proof of Theorem \ref{prop:Ti-Unbounded}. To minimize the complications caused by the presence of $\bdy B_1^+\intersect\bdy \bb R_+^n$, we assume throughout Section \ref{section:Proof of Theorem} that $c$ is constant.

Consider the functions

\begin{equation*}
	v_i(y)
		=
		\frac{1}{M_i} v_i(x_i + \Gamma_i^{-1}(y - T_i e_n))
		=
		\frac{1}{M_i}u_i( x_i' + \Gamma_i^{-1}y),
\end{equation*}
where $x_i' = (x_1, \cdots, x_{n - 1}, 0)$ is the projection of $x_i$ onto $\bb R^{n - 1}$. By the the equations for $v_i$, standard elliptic theory, the selection process and by the classification theorem of Li and Zhu \cite{LiZhu1995}, there is a subsequence along which both $T_i$ converges and $v_i$ converges in $C_{\rm loc}^2(\overline{\bb R_+^n})$ to a classical solution $U$ of \eqref{tibdd}. Letting $c_0 = \lim_i c_i$, the classification theorem of Li and Zhu \cite{LiZhu1995} gives

\begin{equation}\label{tibdd}
	U(y)
		=
		\left(
			\frac{\gamma}
			{\gamma^2 + \abs{ y - t_0e_n}^2}
		\right)^{\frac{n - 2}{2}},
\end{equation}
where

\begin{equation*}
	\gamma
		=
		\left\{
			\begin{array}{ll}
				\left( 1 + \frac{c_0^2}{(n - 2)^2}\right)^{-1}
				&
				\text{ if }
				c_0\leq 0
				\\
				1
				&
				\text{ if }
				c_0>0
			\end{array}
		\right.
	\qquad
	\text{ and }
	\qquad
	t_0
		=
		\frac{\gamma c_0}{n - 2}.
\end{equation*}
Moreover,

\begin{equation*}
	\lim_iT_i
		=
		\left\{
			\begin{array}{ll}
				0
				&
				\text{ if } c_0\leq 0
				\\
				c_0/(n - 2)
				&
				\text{ if }
				c_0>0.
			\end{array}
		\right.
\end{equation*}
We begin by deriving a preliminary vanishing rate for $\abs{\Grad K_i(x_i')}$. For convenience, throughout Section \ref{section:Proof of Theorem} we use the notation $\abs{\Grad K_i(x_i')} = \delta_i$.

\subsection{Vanishing of ${\Grad K_i(x_i')}$}
\label{subsection:Vanishing-Ti-Bounded}
\begin{prop}\label{lem:Vanishing-Ti-Bounded}
	There exists a subsequence along which $\delta_i\to 0$.
\end{prop}

	The proof is similar to the proof of Proposition \ref{lem:Vanishing-Ti-Unbounded}, the major difference being that in this case, a test function must be constructed to control terms in the boundary equation. Suppose the proposition were false and let $\delta>0$ satisfy $\inf_i\delta_i \geq \delta >0$. By assumption \ref{item:K3} we assume with no loss of generality that $\delta_i^{-1}K_i(x_i')\to e$. \\
	
	For $R\gg 1$ fixed and to be determined, we consider the functions
	
	\begin{equation}\label{eq:vRi-Definition-Ti-Bounded}
		v_{R,i}(y)
			 =
			 v_i(y - Re)
			 =
			 \frac{1}{M_i} u_i(x_i' + \Gamma_i (y - Re))
	\end{equation}
	which are well-defined in $\overline{B^+}(0, \Gamma_i/4)$. Similarly to \eqref{eq:vi-Lower-Bound-Upper-Boundary}, we may choose $\epsilon_i\to 0$ slowly so that
	\begin{equation}\label{eq:vRi-Upper-Boundary-Ti-Bounded}
		v_{R,i}(y) \geq \abs y^{2-n}\sqrt i,
		\qquad
		y\in \bdy B(0, \epsilon_i \Gamma_i) \intersect\overline{\bb R_+^n},
	\end{equation}
	so in this case we set
	
	\begin{equation}\label{eq:Omega-Definitions-Ti-Bounded}
		\Omega_i = B^+(0, \epsilon_i\Gamma_i),
		\qquad
		\bdy'\Omega_i = \bdy\Omega_i \intersect \bdy \bb R_+^n
		\qquad
		\text{ and }
		\qquad
		\bdy'' \Omega_i = \bdy \Omega_i \setminus \bdy'\Omega_i.
	\end{equation}
	Elementary computations show that $v_{R,i}$ satisfies
	
	\begin{equation}\label{eq:vRi-Differential-Equations-Ti-Bounded}
		\left\{
			\begin{array}{ll}
				\lap v_{R,i}(y) + H_i(y - Re) v_{R,i}(y)^{n^*} = 0
					&
					y\in \Omega_i
					\\
					\frac{\partial v_{R,i}}{\partial y_n}(y) = c_i v_{R,i}(y)^{n/(n - 2)}
					&
					y\in \bdy'\Omega_i,
			\end{array}
		\right.
	\end{equation}
	where $H_i(y) = K_i(x_i' + \Gamma_i^{-1}y)$. Moreover, $v_{R,i}$ converges in $C^2$ over compact subsets of $\overline{\bb R_+^n}$ to
	
	\begin{equation}\label{eq:Standard-Bubble-Ti-Bounded}
		U_R(y)
			=
			U(y - Re)
			=
			\left(
				\frac{\gamma}
				{\gamma^2 + \abs{ y - Re - t_0e_n}^2}
			\right)^{\frac{n - 2}{2}}.
	\end{equation}
	For $\lambda >0$ let $y^\lambda = \lambda^2 y/\abs y^2$ and consider the Kelvin inversions
	
	\begin{equation}\label{eq:Kelvin-Inversions-Ti-Bounded}
		U_R^\lambda(y)
			=
			\left(\frac{\lambda}{\abs y}\right)^{n - 2}
			U_R(y^\lambda),
		\qquad
		\text{ and }
		\qquad
		v_{R,i}^\lambda(y)
			=
			\left(\frac{\lambda}{\abs y}\right)
			v_{R,i}(y^\lambda).
	\end{equation}
	which are well-defined in the closure of $\Sigma_\lambda = \Omega_i \setminus B_\lambda$. Letting $\lambda^* = (\gamma^2 + t_0^2 + R^2)^{1/2}$, elementary computations show that
	
	\begin{equation}\label{eq:Critical-Position-Inequalities-Ti-Bounded}
		\left\{
			\begin{array}{lll}
				(U_R - U_R^\lambda)(y) >0
				&
				y\in \Sigma_\lambda
				&
				\text{ if }\lambda <\lambda^*
				\\
				(U_R - U_R^\lambda)(y) <0
				&
				y\in \Sigma_\lambda
				&
				\text{ if }\lambda >\lambda^*.
			\end{array}
		\right.
	\end{equation}
	Set $\lambda_0 = R$ and $\lambda_1 = R+2$. Since $\lambda_0 < \lambda^*<\lambda_1$  we only consider $\lambda$ between $\lambda_0$ and $\lambda_1$. For such $\lambda$, define
	
	\begin{equation}\label{eq:wLambda-Definition-Ti-Bounded}
		w^\lambda(y)
			=
			w_i^\lambda(y)
			=
			v_{R,i}(y) - v_{R,i}^\lambda(y),
		\qquad
		y\in \overline\Sigma_\lambda.
	\end{equation}
	For convenience we suppress both the $i$-dependence and the $R$-dependence in this notation.  Elementary computations show that $w^\lambda $ satisfies
	
	\begin{equation}\label{eq:wLambda-Equations-Ti-Bounded}
		\left\{
			\begin{array}{ll}
				L_i w^\lambda = Q^\lambda
				&
				y\in \Sigma_\lambda
				\\
				B_i w^\lambda = 0
				&
				y\in \bdy'\Sigma_\lambda
				\\
				w^\lambda(y) = 0
				&
				y\in \bdy\Sigma_\lambda \intersect \bdy B_\lambda,
			\end{array}
		\right.
	\end{equation}
	where
	
	\begin{equation*}
		\begin{array}{l}
			L_i = \lap + H_i(y - Re)\xi_1(y)
			\\
			B_i = \frac{\partial}{\partial y_n} - c_i \xi_2(y)
		\end{array}
	\end{equation*}
	are the interior and boundary operators respectively,
	
	\begin{equation}\label{eq:Mean-Value-Coefficient-1}
		\xi_1(y)
			=
			n^* \int_0^1 \left( t v_{R,i}(y) + (1 - t) v_{R,i}^\lambda(y) \right)^{\frac{4}{n -2}}\; dt
	\end{equation}
	\begin{equation}\label{eq:Mean-Value-Coefficient-2}
		\xi_2(y)
			=
			\frac{n}{n - 2} \int_0^1 \left( t v_{R,i}(y) + (1 - t) v_{R,i}^\lambda(y) \right)^{\frac{2}{n -2}}\; dt
	\end{equation}
	are obtained from the mean-value theorem and
	
	\begin{equation*}
		Q^\lambda(y)
			=
			\left(
			H_i(y^\lambda - Re) - H_i(y - Re)
			\right)
			\left( v_{R,i}(y)\right)^{n^*}
	\end{equation*}
	is an error term that will be controlled with test functions. Specifically, we will construct a test function $h^\lambda(y)$ such that both
	
	\begin{equation}\label{eq:hLambda-o(1)}
		h^\lambda(y)= \circ(1) \abs y^{2-n}
	\end{equation}
	and
	
	\begin{equation}\label{eq:Desired-Differential-Inequalities-Ti-Bounded}
		\left\{
			\begin{array}{ll}
				L_i(w^\lambda + h^\lambda)(y) \leq 0
				&
				y\in \Sigma_\lambda \intersect \mathcal O_\lambda
				\\
				B_i(w^\lambda + h^\lambda)(y)\leq 0
				&
				y\in \bdy'\Sigma_\lambda \intersect \overline{\mathcal O}_\lambda
				\\
				(w^\lambda + h^\lambda)(y) = 0
				&
				y\in \bdy\Sigma_\lambda \intersect \bdy B_\lambda,
			\end{array}
		\right.
	\end{equation}
	where
	
	\begin{equation*}
		\mathcal O_\lambda
			=
			\{
				y\in \Sigma_\lambda:
				(v_{R,i} - v_{R,i}^\lambda)(y) \leq v_{R,i}^\lambda(y)
			\}.
	\end{equation*}
	This will allow the maximum principle to be applied. Note that the maximum principle only needs to hold on $\mathcal O_\lambda$. This is because of \eqref{eq:hLambda-o(1)}; if $i$ is sufficiently large, then $(w^\lambda + h^\lambda)>0$ in $\Sigma_\lambda \setminus \mathcal O_\lambda$. \\
	
	Before we construct $h^\lambda$ we record some estimates that will be useful when deriving properties of $h^\lambda$ after it is constructed. We define the special subset of $\Sigma_\lambda$
	
	\begin{equation*}
		\Omega_\lambda
			=
			\{
				y\in \Sigma_\lambda \intersect\overline B_{2\lambda}:
				y_1>2\abs{(y_2, \cdots, y_n)}
			\}.
	\end{equation*}
	By the assumptions on $K$ and the convergence of $v_{R,i}$ to $U_R$ we have
		
		\begin{equation*}
			\left\{
				\begin{array}{ll}
					H_i(y^\lambda - Re) - H_i(y - Re)
						\leq
						- C_1 \Gamma_i^{-1} (\abs y - \lambda)
					&
					y\in \Omega_\lambda
					\\
					\abs{H_i(y^\lambda - Re) - H_i(y - Re)}
						\leq
						C_2 \Gamma_i^{-1} (\abs y - \lambda)
					&
					y\in \Sigma_\lambda.
				\end{array}
			\right.
		\end{equation*}
		Moreover, similarly to Lemma \ref{lem:vRi-Estimates}, there are positive constants $C_1, C_2$ such that for large $i$, both
		
		\begin{equation*}
			C_1\abs y^{2-n} \leq v_{R,i}^\lambda(y)\leq C_2\abs y^{2-n}
			\qquad y\in \Sigma_\lambda \setminus \Omega_\lambda,
		\end{equation*}
		and
		
		\begin{equation*}
			C_1 \left(\frac \lambda{\abs y}\right)^{n - 2}
				\left(\frac{1}{1 + \abs{ y - \lambda e}^2}
				\right)^{(n - 2)/2}
			\leq v_{R,i}^\lambda(y)
			\leq 2
			\qquad y \in \Omega_\lambda.
		\end{equation*}
	
	Therefore, there are positive $\lambda$-independent constants $a_1, a_2$ such that
	
	\begin{equation}\label{eq:QLambdaNegEstimate}
		Q^\lambda(y)
			\leq - a_1\Gamma_i^{-1} (\abs y - \lambda)
				\left(\frac{1}{1 +\abs{y - \lambda e}^2}
				\right)^{(n + 2)/2}
				\qquad
				y\in \Omega_\lambda.
	\end{equation}
	and
	
	\begin{equation}\label{eq:AbsQLambdaEstimates}
		\abs{Q^\lambda(y)} \leq
			\left\{
				\begin{array}{ll}
					a_2\Gamma_i^{-1} (\abs y - \lambda) \abs y^{-2-n}
						& y\in \Sigma_\lambda \setminus
						(\overline{B_\lambda \union \Omega_\lambda})\\
					a_2\Gamma_i^{-1} (\abs y - \lambda)
						& y \in \overline \Omega_\lambda.
				\end{array}
			\right.
	\end{equation}
	
	Finally, $\xi_1$ and $\xi_2$ still satisfy the conclusions of Lemma \ref{lem:Mean-Value-Coefficient-Estimates-Ti-Unbounded}.\\
	
	 Now we construct the test function $h^\lambda$ which will be the sum of two functions $h_1$ and $h_2$. The first test function $h_1$ is similar to the test function constructed in \eqref{eq:Naive-hLambda-Definition}. The second test function $h_2$ will control the bad terms on $\bdy'\Sigma_\lambda$ introduced by $h_1$.
	%
	Let $G(y, \eta)$ be Green's function for $-\lap$ on $\bb R^n\setminus \overline B_\lambda$ relative to the Dirichlet condition. The expression for $G(y, \eta)$ is given in \eqref{eq:Greens-Function-Formula}.
	Let $\bar y = (y_1, \cdots, y_{n - 1}, -y_n)$ denote the reflection of $y$ across $\bdy \bb R_+^n$ and set
	
	\begin{equation*}\label{eq:Reflected-Greens-Function}
		\overline G(y, \eta) = G(y, \eta) + G(\bar y, \eta).
	\end{equation*}
	 Define
	
	\begin{equation*}
		h_1(y)
		=
		\int_{\Sigma_\lambda}
			\overline G(y, \eta) Q^\lambda(\eta)
			\; d\eta.
	\end{equation*}
	Clearly $h_1$ satisfies the following
	
	\begin{equation}\label{eq:h1-Preliminary-Estimates}
	\left\{
		\begin{array}{ll}
			-\lap h_1(y) = Q^\lambda(y)
				& y\in \Sigma_\lambda
				\\
			h_1(y) = 0
				&
				y\in \bdy\Sigma_\lambda \intersect \bdy B_\lambda
				\\
			\frac{\partial h_1}{\partial y_n} \ident 0
				&
				y\in \bdy' \Sigma_\lambda.
		\end{array}
	\right.
	\end{equation}
	\begin{lem}\label{lem:h1Estimates}
	There exists $R_0$ sufficiently large such that if $R\geq R_0$ then there are positive constants $C_1$ and $C_2$ such that
	
	\begin{equation}\label{eq:h1Estimates}
		h_1(y) \leq
			\left\{
				\begin{array}{ll}
					- C_1 \Gamma_i^{-1}(\abs y - \lambda)\lambda^{-n} \log \lambda
						& y\in \overline{\Sigma_\lambda} \intersect \overline B_{4\lambda}\\
					- C_1 \Gamma_i^{-1}\abs y^{2-n} \lambda^{-1}\log \lambda
						& y\in \overline{\Sigma_\lambda} \setminus B_{4\lambda}
				\end{array}
			\right.
	\end{equation}
	
	and
	
	\begin{equation}\label{eq:Absh1Estimates}
		\abs{h_1(y)} \leq
			\left\{
				\begin{array}{ll}
					C_2\Gamma_i^{-1}(\abs y - \lambda)\lambda^2
						& y\in \overline{\Sigma_\lambda} \intersect \overline B_{4\lambda}\\
					C_2\Gamma_i^{-1}\abs y^{2-n} \lambda^{n+1}
						& y\in \overline{\Sigma_\lambda} \setminus B_{4\lambda}.
				\end{array}
			\right.
	\end{equation}
	\end{lem}
	The proof of Lemma \ref{lem:h1Estimates} is similar to the proof of Lemma \ref{lem:hLambda-Estimates-Ti-Unbounded} and is omitted.
%
%
	
	Now we define the second test function $h_2$. Let $g:[\lambda, \infty)\to [0,\infty)$ be a smooth positive function satisfying
	
	\begin{equation}\label{eq:gDefinition}
		g(r) =
			\left\{
				\begin{array}{ll}
					\lambda\left(\frac{r}{\lambda}\right)^n - \frac{n - 1}{2}\frac{r^2}{\lambda}
					+
					\frac{\lambda}{2}(n - 3)
						& \lambda \leq r\leq 3\lambda\\
					\text{ smooth positive connection }
						& 3\lambda\leq r\leq 4\lambda\\
					\lambda^{-1} -  r^{-1}
						& 4\lambda \leq r,
				\end{array}
			\right.
	\end{equation}
	where `smooth positive connection' means there is a constant $M(\lambda)>0$ such that both
	
	\begin{equation*}
		\norm g_{C^2([3\lambda, 4\lambda])} \leq M
	\end{equation*}
	and
	
	\begin{equation}\label{eq:g-Middle-Estimate}
		g(r) \geq \frac 1 M
		\qquad
		3\lambda \leq r \leq 4\lambda.
	\end{equation}
	Elementary computations show that
	
	\begin{equation*}
		\left\{
			\begin{array}{ll}
				g(\lambda) = 0,
				&
				g'(\lambda) = 1
				\\
				g''(r) >0,
				 &
				\lambda <r<3\lambda.
			\end{array}
		\right.
	\end{equation*}
	In particular, $g'(r)>1$ for $\lambda <r\leq 3\lambda$ so there is a positive constant $C$ such that
	
	\begin{equation}\label{eq:g-Linear-Estimate}
		r - \lambda
		\leq
		g(r)
		\leq
		C(r - \lambda)
		\qquad
		\lambda \leq r \leq 3\lambda.
	\end{equation}
	Moreover, we have both
	
	\begin{equation}\label{eq:g''-Estimates-1}
		g''(r) - \frac{n - 1}{r} g'(r)
		=
		\left\{
			\begin{array}{ll}
				(n - 1)(n - 2)\lambda^{-1}
				&
				\lambda <r<3\lambda
				\\
				-(n+1)r^{-3}
				&
				4\lambda<r
			\end{array}
		\right.
	\end{equation}
	and
	
	\begin{equation}\label{eq:g''-Estimates-2}
		- M \leq
		g''(r) - \frac{n - 1}{r} g'(r)
		\leq
		M
		\qquad
		3\lambda \leq r\leq 4\lambda.
	\end{equation}
	For $a>0$ fixed but to be determined (will be chosen sufficiently small and depending on $n, \Lambda, \lambda$ and $M$) define
	
	\begin{equation*}
	h_2(y) = - a\Gamma_i^{-1}y_n \abs y^{-n}g(\abs y)
	\qquad
	y\in \overline{\Sigma}_\lambda.
	\end{equation*}
	Clearly, $h_2<0$ in $\Sigma_\lambda$, $h_2\ident 0$ on $\bdy\Sigma_\lambda \intersect (\bdy B_\lambda \union \bdy \bb R_+^n)$ and
	
	\begin{eqnarray}\label{eq:Absh2Estimate}
		\abs{h_2(y)}
			 & \leq &
			 	\left\{
					\begin{array}{ll}
						a \Gamma_i^{-1}\lambda^{1-n} (\abs y - \lambda)
						&
						y\in \overline{\Sigma_\lambda}\intersect \overline{B_{3\lambda}}
						\\
						a M \Gamma_i^{-1} \lambda^{1-n}
						&
						y\in (\overline{\Sigma_\lambda}\intersect \overline B_{4\lambda})
						\setminus
						B_{3\lambda}
						\\
						a \Gamma_i^{-1}\abs y^{1-n}\lambda^{-1}
						&
						y\in \overline{\Sigma_\lambda}\setminus B_{4\lambda}
					\end{array}
				\right.
			\\
			& = &
			\circ(1) \abs y^{2-n}.\notag
	\end{eqnarray}
	Performing elementary computations and using the properties of $g$ given in \eqref{eq:g''-Estimates-1} and \eqref{eq:g''-Estimates-2} we obtain
	
	\begin{eqnarray}\label{eq:Laplacian-h2-Estimates}
		\lap h_2(y)
			& = &
			- a \Gamma_i^{-1} y_n\abs y^{-n}
			\left(
				g''(\abs y) - \frac{n - 1}{\abs y} g'(\abs y)
			\right)
			\notag
			\\
			& \leq &
				\left\{
					\begin{array}{ll}
						0
						&
						y\in \Sigma_\lambda \intersect B_{3\lambda}
						\\
						\bar a \Gamma_i^{-1} M\lambda^{1-n}
						&
						y\in (\Sigma_\lambda \intersect B_{4\lambda})
						\setminus B_{3\lambda}
						\\
						\bar a \Gamma_i^{-1} \abs y^{-2-n}
						&
						y\in \Sigma_\lambda \setminus B_{4\lambda},
					\end{array}
				\right.
	\end{eqnarray}
	where $\bar a$ denotes a constant of the form $C(n) a$. Also, using  \eqref{eq:gDefinition}, \eqref{eq:g-Middle-Estimate} and \eqref{eq:g-Linear-Estimate} we obtain
	
	\begin{eqnarray}\label{eq:h2-Boundary-Derivative-Estimate}
		\left.\frac{\partial h_2}{\partial y_n}(y) \right|_{y\in \bdy'\Sigma_\lambda}
			& = &
			- a \Gamma_i^{-1} \abs y^{-n} g(\abs y)
			\notag
			\\
			& \leq &
			\left\{
				\begin{array}{ll}
					- \bar a \Gamma_i^{-1} \lambda^{-n}(\abs y - \lambda)
					&
					y\in \bdy'\Sigma_\lambda \intersect \overline B_{3\lambda}
					\\
					- \bar a M^{-1} \Gamma_i^{-1} \lambda^{-n}
					&
					y\in (\bdy'\Sigma_\lambda \intersect \overline B_{\lambda})
					\setminus B_{3\lambda}
					\\
					-\bar a \Gamma_i^{-1} \lambda^{-1}\abs y^{-n}
					&
					y\in \bdy'\Sigma_\lambda \setminus B_{4\lambda}.
				\end{array}
			\right.
	\end{eqnarray}
%
	Let $h^\lambda = h_1 + h_2$. Since each of $h_1$ and $h_2$ are non-positive, using \eqref{eq:h1-Preliminary-Estimates} we obtain
	
	\begin{equation}\label{eq:Differential-Inequalities}
		\left\{
			\begin{array}{ll}
				L_i (w^\lambda + h^\lambda)(y)
					\leq
					H_i(y - Re) \xi_1(y) h_1(y) + \lap h_2
				&
				y\in \Sigma_\lambda
				\\
				B_i(w^\lambda + h^\lambda)(y)
					\leq c_i \xi_2(y) \abs{ h_1(y)} + \frac{\partial h_2}{\partial y_n}(y)
				&
				y\in \bdy' \Sigma_\lambda
				\\
				(w^\lambda + h^\lambda)(y)
					=
					0
				&
				y\in \bdy\Sigma_\lambda \intersect \bdy B_\lambda.
			\end{array}
		\right.
	\end{equation}
	Moreover, since $H_i(y - Re) \geq \Lambda^{-1}$, using Lemma \ref{lem:Mean-Value-Coefficient-Estimates-Ti-Unbounded}, equation \eqref{eq:h1Estimates} and \eqref{eq:Laplacian-h2-Estimates} we see that $a= a(M,\lambda)$ may be chosen sufficiently small to achieve
	
	\begin{equation*}
		L_i(w^\lambda + h^\lambda)(y) \leq 0
		\qquad
		y\in \Sigma_\lambda.
	\end{equation*}
	Now consider the boundary inequality in \eqref{eq:Differential-Inequalities}. If $c_i\leq 0$ then $B_i(w^\lambda  +h^\lambda)\leq 0$ on $\bdy'\Sigma_\lambda$ holds trivially as $\partial h_2/\partial y_n \leq0$. We only need to consider the case $c_i>0$. By \eqref{eq:Absh1Estimates} and \eqref{eq:Absh2Estimate} there is a constant $C(M,\lambda)>0$ such that
	
	\begin{eqnarray*}
		\abs{h_1(y)}
			& \leq &
			\left\{
				\begin{array}{ll}
					C \Gamma_i^{-1}(\abs y - \lambda)
					&
					y\in \bdy'\Sigma_\lambda \intersect \overline B_{3\lambda}
					\\
					C\Gamma_i^{-1}
					&
					y\in (\bdy'\Sigma_\lambda \intersect \overline B_{\lambda})
					\setminus B_{3\lambda}
					\\
					C \Gamma_i^{-1} \abs y^{2-n}
					&
					y\in \bdy'\Sigma_\lambda \setminus B_{\lambda}.
				\end{array}
			\right.	
	\end{eqnarray*}
	Combining this with lemma \ref{lem:Mean-Value-Coefficient-Estimates-Ti-Unbounded} and \eqref{eq:h2-Boundary-Derivative-Estimate} we see that there is $\epsilon(n, \Lambda, \lambda, M, a)>0$ such that if $c_0<\epsilon$ then
	
	\begin{equation*}
		B_i(w^\lambda + h^\lambda)(y)
			\leq
			0
		\qquad
		y\in \bdy'\Sigma_\lambda \intersect \overline{\mathcal O}_\lambda.
	\end{equation*}
	%
%
	The next lemma ensures that the moving sphere process can start.
	\begin{lem}\label{lem:Start-Moving-Spheres-Ti-Bounded}
		There exists $\epsilon> 0$ sufficiently small and $i_0\in \bb N$ such that if $c_0<\epsilon$ and $i\geq i_0$ then
		
		\begin{equation*}
			w^{\lambda_0}(y) + h^{\lambda_0}(y) >0
			\qquad
			y\in \Sigma_{\lambda_0}.
		\end{equation*}
	\end{lem}
	\begin{proof}
		If $R_1\gg R$ is any fixed large constant, the for $i$ sufficiently large, $w^{\lambda_0} + h^{\lambda_0}>0$ in $\Sigma_{\lambda_0}\intersect B_{R_1}$. This is because of the properties of $U_R - U_R^{\lambda_0}$, the convergence of $w^{\lambda_0}$ to $U_R - U_R^{\lambda_0}$ and the estimate $h^{\lambda_0} = \circ(1) \abs y^{2-n}$. We only need to show positivity of $w^{\lambda_0} + h^{\lambda_0}$ on $\Sigma_{\lambda_0} \setminus B_{R_1}$. \\
		By performing elementary estimates it is easy to see that there exists $\epsilon_0(\gamma, t_0, \lambda_0)>0$ such that 		
		
		\begin{equation*}
			U_R^{\lambda_0}(y)
				\leq (1 - 5\epsilon_0) \gamma^{\frac{n - 2}{2}}\abs{ y - e_n}^{2-n},
				\qquad
				\abs y\geq R_1.
		\end{equation*}
		By increasing $R_1$ if necessary, we may simultaneously achieve
		\begin{equation}\label{eq:Start-Moving-Spheres-UR-Lower-Bound-Ti-Bounded}
			U_R(y)\geq \left(1 - \frac{\epsilon_0}{2}\right)
			\gamma^{\frac{n - 2}{2}}\abs{ y - e_n}^{2-n},
			\qquad
			\abs y = R_1.
		\end{equation}
		As an immediate consequence of these inequalities and the convergence of $v_{R,i}$ to $U_R$, if $i$ is sufficiently large we have
		
		\begin{equation}\label{eq:vRLambda0UpperBound}
			v_{R,i}^{\lambda_0}(y)
				\leq
				(1 - 4\epsilon_0) \gamma^{\frac{n - 2}{2}} \abs{y - e_n}^{2-n}
				\qquad
				y\in \Sigma_{\lambda _0} \setminus B_{R_1}.
		\end{equation}
		Now  suppose
		
		\begin{equation}\label{eq:c0SmallnessAssumption}
			c_0< (n - 2)(2\gamma)^{-1}(1 - \epsilon_0)^{-2/(n - 2)}.
		\end{equation}
		We show that if $i$ is sufficiently large, then
			
			\begin{equation}\label{eq:vRLowerBound}
				v_{R,i}^{\lambda_0}(y)
					>
					(1 - \epsilon_0) \gamma^{\frac{n - 2}{2}} \abs{ y - e_n}^{2-n}
					\qquad
					y\in \Sigma_{\lambda_0}\setminus B_{R_1}.
			\end{equation}
	By \eqref{eq:Start-Moving-Spheres-UR-Lower-Bound-Ti-Bounded} and the convergence of $v_{R,i}$ to $U_R$, if $i$ is sufficiently large, then
			
			\begin{equation*}
				v_{R,i}(y) >(1 - \epsilon_0) \gamma^{\frac{n - 2}{2}} \abs{ y - e_n}^{2-n}
				\qquad
				y\in \Sigma_{\lambda_0} \intersect \bdy B_{R_1}.
			\end{equation*}
			Therefore,
			
			\begin{equation*}
				f_i(y) = v_{R,i}(y) - (1 - \epsilon_0)\gamma^{\frac{n - 2}{2}}\abs{ y - e_n}^{2-n}
			\end{equation*}
			is superharmonic in $\Sigma_{\lambda_0} \setminus B_{R_1}$ and positive on $\overline{\Sigma_{\lambda_0}}\intersect\bdy B_{R_1}$. Moreover, by \eqref{eq:vRi-Upper-Boundary-Ti-Bounded}, if $i$ is sufficiently large,
			
			\begin{equation*}
				f_i(y) \geq C(n,\Lambda) \sqrt i \abs {y- e_n}^{2-n}
				\qquad
				y\in \bdy \Sigma_{\lambda_0} \intersect \{\abs y = \epsilon_i \Gamma_i\}.
			\end{equation*}
			By the maximum principle, if $f_i$ achieves a nonpositive minimum on $\overline{\Sigma_{\lambda_0}}\setminus B_{R_1}$, it must occur on $\bdy '\Sigma_{\lambda_0}\setminus B_{R_1}$. However, this is impossible. Indeed, suppose $y_i^*\in \bdy'\Sigma_{\lambda_0} \setminus B_{R_1}$ satisfies
			
			\begin{equation}\label{eq:NonpositiveMin}
				\min_{\overline{\Sigma_{\lambda_0}}\setminus B_{R_1}} f_i(y)
				=
				f_i(y_i^*)
				\leq 0.
			\end{equation}
			Since $y_{in}^* = 0$ we have
			
			\begin{equation}\label{eq:StartMovingSpheresBoundaryEquation}
				\frac{\partial f_i}{\partial y_n} (y_i^*)
					=
					c_iv_i(y_i^*)^{\frac{n}{n - 2}}
					- (n - 2)(1 - \epsilon_0)\gamma^{\frac{n - 2}{2}} \abs{ y_i^* - e_n}^{-n}.
			\end{equation}
			If $c_0\leq 0$ either $c_i<0$ or both $c_i\geq 0$ and $c_i = \circ(1)$. If $c_i<0$ then $\frac{\partial f}{\partial y_n}(y_i^*)<0$. If $0\leq c_i = \circ(1)$ then by \eqref{eq:NonpositiveMin} and \eqref{eq:StartMovingSpheresBoundaryEquation}, if $i$ is sufficiently large then $\frac{\partial f_i}{\partial y_n}(y_i^*)<0$. Finally, if $c_0>0$ the using \eqref{eq:NonpositiveMin} once more along with the smallness assumption \eqref{eq:c0SmallnessAssumption} we obtain $\frac{\partial f}{\partial y_n}(y_i^*)<0$ for $i$ sufficiently large. In any case, $\frac{\partial f_i}{\partial y_n}(y_i^*)<0$ so $y_i^*$ is not a minimizer for $f_i$.
	\end{proof}

\begin{proof}[Proof of Proposition \ref{lem:Vanishing-Ti-Bounded}]
With Lemma \ref{lem:Start-Moving-Spheres-Ti-Bounded} proven, the moving sphere process can start at $\lambda = \lambda_0$.  Since $h^\lambda$ satisfies \eqref{eq:hLambda-o(1)} and \eqref{eq:Desired-Differential-Inequalities-Ti-Bounded} , we can show that for

\begin{equation*}
	\bar\lambda
		=
		\sup\{ \lambda \in [\lambda_0, \lambda_1]:
		w^\mu(y) + h^\mu(y) \geq0 \; \text{ in } \; \Sigma_{\mu}\:
		\text{ for all } \lambda_0\leq \mu\leq \lambda_1\},
\end{equation*}
we have $\bar \lambda = \lambda_1$. This contradicts the convergence of $v_{R,i}$ to $U_R$.
\end{proof}

\subsection{Improved Vanishing Rate for $\abs{\Grad K_i(x_i')}$}
\label{subsection:Fast-Vanishing-Ti-Bounded}

In this subsection we derive a fast vanishing rate for $\delta_i$.

\begin{prop}\label{lem:Fast-Vanishing-Ti-Bounded}
	There exists a constant $C>0$ such that
	
	\begin{equation*}
		\Gamma_i^{n - 3} \delta_i \leq C.
	\end{equation*}
\end{prop}

 The proof of Proposition \ref{lem:Fast-Vanishing-Ti-Bounded} is similar in spirit to the proof of Proposition \ref{lem:Fast-Vanishing-Ti-Unbounded}. The difference is that the proof of Proposition \ref{lem:Fast-Vanishing-Ti-Bounded} requires a second test function to control an unfavorable boundary term introduced by the first test function.

 \begin{proof}
As in the proof of Proposition \ref{lem:Fast-Vanishing-Ti-Unbounded}, the proof of Proposition \ref{lem:Fast-Vanishing-Ti-Bounded} is by contradiction and we pass to a subsequence for which both

 \begin{equation*}
 	\ell_i\to\infty
	\qquad
	\text{ and }
	\qquad
	\delta_i^{-1} \Grad K_i(x_i') \to e.
 \end{equation*}
 For $R\gg1$ fixed and to be determined, let $v_{R,i}$ be as in \eqref{eq:vRi-Definition-Ti-Bounded} and let $\Omega_i$, $\bdy'\Omega_i$ and $\bdy''\Omega_i$ be as in \eqref{eq:Omega-Definitions-Ti-Bounded}. As in the proof of Proposition \ref{lem:Vanishing-Ti-Bounded}, $v_{R,i}$ satisfies both \eqref{eq:vRi-Upper-Boundary-Ti-Bounded} and \eqref{eq:vRi-Differential-Equations-Ti-Bounded} and converges to $U_R(y)$ in $C^2$ over compact subsets of $\overline{\bb R_+^n}$, where $U_R$ is given by \eqref{eq:Standard-Bubble-Ti-Bounded}. Letting $U_R^\lambda$ and $v_{R,i}^\lambda$ denote the Kelvin inversions of $U_R$ and $v_{R,i}$ as in \eqref{eq:Kelvin-Inversions-Ti-Bounded}, we still have \eqref{eq:Critical-Position-Inequalities-Ti-Bounded}. We only consider $\lambda$ between $\lambda_0 = R$ and $\lambda_1 = R+2$. Letting $w^\lambda$ be as in \eqref{eq:wLambda-Definition-Ti-Bounded}, we still have \eqref{eq:wLambda-Equations-Ti-Bounded}, so we need to construct $h^\lambda$ that satisfies both \eqref{eq:hLambda-o(1)} and \eqref{eq:Desired-Differential-Inequalities-Ti-Bounded}. We start with some helpful estimates.

 \begin{lem}\label{lem:Fast-Scalar-Curvature-Estimates-Ti-Bounded}
 	There exist positive constants $C_1$ and $C_2$ such that for $i$ sufficiently large, both
	
	\begin{equation*}
		H_i(y^\lambda - Re) - H_i(y - Re)
			\leq
			-C_1 \Gamma_i^{-1} \delta_i(\abs y - \lambda)
			\qquad
			y\in
			\Omega_\lambda.
	\end{equation*}
	and
	
	\begin{equation*}
		\abs{H_i(y^\lambda - Re) - H_i(y - Re)}
			\leq
			\left\{
				\begin{array}{ll}
					C_2\Gamma_i^{-1} \delta_i (\abs y - \lambda)
					&
					y\in\Omega_\lambda
					\\
					C_2\Gamma_i^{-1} \delta_i (\abs y - \lambda)
					\sum_{j = 0}^{n - 3} \ell_i^{-j} \abs y^j
					&
					y\in \Sigma_\lambda\setminus\Omega_\lambda.
				\end{array}
			\right.
	\end{equation*}
 \end{lem}
 The proof of Lemma \ref{lem:Fast-Scalar-Curvature-Estimates-Ti-Bounded} is similar to the proof of Lemma \ref{lem:Scalar-Curvature-Function-Fast-Estimates} and is omitted. \\

By Lemma \ref{lem:Fast-Scalar-Curvature-Estimates-Ti-Bounded} and Lemma \ref{lem:vRi-Estimates}, we obtain positive $\lambda$-independent constants $a_1$ and $a_2$ such that both

 \begin{equation*}
 	Q^\lambda(y)
		\leq
		- a_1 \Gamma_i^{-1} \delta_i(\abs y - \lambda)
		(1 + \abs{ y - \lambda e}^2)^{-\frac{n+2}{2}},
		\qquad
		y\in \Omega_\lambda
 \end{equation*}
 and
 \begin{equation*}
 	\abs{Q^\lambda(y) }
		\leq
		a_2 \Gamma_i^{-1} \delta_i(\abs y - \lambda)
		\sum_{j = 0}^{n - 3} \ell_i^{-j} \abs y^{j - 2-n},
		\qquad
		y\in \Sigma_\lambda.
 \end{equation*}
Let $\Hat Q^\lambda$ be as in \eqref{eq:QLambdaHat-Definition}. The estimates in \eqref{eq:QLambdaHat-Upper-Bounds}, \eqref{eq:Abs-QLambdaHat-Estimate} and \eqref{eq:QLambda-QLambdaHat} are still satisfied. Define

 \begin{equation*}
 	h_1(y)
		=
		\int_{\Sigma_\lambda} \overline G(y, \eta) \Hat Q^\lambda(\eta) \; d\eta,
		\qquad
		y\in \Sigma_\lambda.
 \end{equation*}
 Then $h_1$ satisfies

 \begin{equation*}
 	\left\{
		\begin{array}{ll}
		 	- \lap h_1(y) = \Hat Q^\lambda(y)
			&
			y\in \Sigma_\lambda
			\\
			h_1(y) = 0
			&
			y\in \bdy\Sigma_\lambda \intersect\bdy B_\lambda
			\\
			\frac{\partial h_1}{\partial y_n}(y) = 0
			&
			y\in \bdy'\Sigma_\lambda.
		\end{array}
	\right.
 \end{equation*}
 As in the proof of Lemma \ref{lem:Fast-hLambda-Estimates}, we still have positive constants $C_1$ and $C_2$ such that both

 \begin{equation}\label{eq:Fast-h1-Estimates-Ti-Bounded}
 	h_1(y)
		\leq
		\left\{
			\begin{array}{ll}
				- C_1 \Gamma_i^{-1}\delta_i(\abs y - \lambda)\lambda^{-n} \log \lambda
				&
				y\in \overline \Sigma_\lambda \intersect \overline B_{4\lambda}
				\\
				C_2 \Gamma_i^{-1} \delta_i \abs y^{2-n}
				\left(
					\ell_i^{-1} \log \frac{\abs y }{\lambda}
					+
					\sum_{j = 2}^{n - 3} \ell_i^{-j} \abs y^{j-1}
				\right)
				&
				y\in \overline \Sigma_\lambda \setminus B_{4\lambda}
			\end{array}
		\right.
 \end{equation}
 and
 \begin{eqnarray}\label{eq:Fast-Abs-h1-Estimate-Ti-Bounded}
 	\abs{h_1(y)}
		& \leq &
		\left\{
			\begin{array}{ll}
				C_2\Gamma_i^{-1} \delta_i \lambda^2 (\abs y - \lambda)
				&
				y\in \overline \Sigma_\lambda \intersect \overline B_{4\lambda}
				\\
				C_2\Gamma_i^{-1} \delta_i \abs y^{2-n}
				\left(
					\lambda^{1+n} + \ell_i^{-1} \log\frac{\abs y}{\lambda}
					+\sum_{j = 2}^{n - 3} \ell_i^{-j} \abs y^{j-1}
				\right)
				&
				y\in \overline \Sigma_\lambda \setminus B_{4\lambda}
			\end{array}
		\right.
		\notag
		\\
		& = &
		\circ(1) \abs y^{2-n}.
 \end{eqnarray}
 %

 For the construction of $h_2$, the second part of the test function, we consider separately the case $c_0<0$ and the case $c_0\geq 0$. \\
 \emph{Case 1:} $c_0<0$. \\
 In this case for $i$ large we have $c_i<0$. Let $g_i:[\lambda,\infty)\to [0,\infty)$ be given by

 \begin{equation*}
 	g_i(r)
		=
		\left\{
			\begin{array}{ll}
				\lambda^{1-n} r^n - \frac{n - 1}{2\lambda} r^2 + \frac \lambda 2(n - 3)
				&
				\lambda \leq r\leq 3\lambda
				\\
				\text{smooth positive connection}
				&
				3\lambda \leq r\leq 4\lambda
				\\
				\log \frac r \lambda + \sum_{j = 2}^{n -3} \ell_i^{1-j} r^{j -1}
				&
				4\lambda \leq r,
			\end{array}
		\right.
 \end{equation*}
 where `smooth positive connection' means there is a positive constant $M(n, \Lambda, \lambda)$ such that both $g_i(r) \geq \frac 1 M$ for $3\lambda \leq r \leq 4\lambda$ and $\norm{g_i}_{C^2([3\lambda,4\lambda])}\leq M$. By elementary estimates we have

 \begin{equation*}
 	g_i''(r) - \frac{n - 1}{r} g_i'(r)
		\geq
		\left\{
			\begin{array}{ll}
				0
				&
				\lambda \leq r\leq 3\lambda
				\\
				- M
				&
				3\lambda \leq r \leq 4\lambda
				\\
				- C\sum_{j = 1}^{n - 3} \ell_i^{1-j} r^{j -3}
				&
				4\lambda \leq r.
			\end{array}
		\right.
 \end{equation*}
 Set

 \begin{equation*}
 	h_2(y)
		=
		-a \Gamma_i^{-1} \delta_i \ell_i^{-1} y_n \abs y^{-n} g_i(\abs y)
		\leq
		0
		\qquad
		\abs y \geq \lambda,
 \end{equation*}
 where $a$ is a positive constant which is to be determined. By direct computation and using the properties of $g_i$ we have both

 \begin{eqnarray}\label{eq:Laplacian-h2-Estimate-c<0}
 	\lap h_2(y)
		& = &
		- a \Gamma_i^{-1} \delta_i \ell_i^{-1} y_n \abs y^{-n}
		\left(
			g_i''(\abs y ) - \frac{n - 1}{\abs y} g_i'(\abs y)
		\right)
		\notag
		\\
		& \leq &
			\left\{
				\begin{array}{ll}
					0
					&
					\lambda \leq \abs y \leq 3\lambda
					\\
					\bar a M \Gamma_i^{-1} \delta_i \ell_i^{-1} \lambda^{1-n}
					&
					3\lambda \leq \abs y \leq 4\lambda
					\\
					\bar a \Gamma_i^{-1} \delta_i \abs y^{-2-n}
					\sum_{j = 1}^{n - 3} \ell_i^{-j} \abs y^j
					&
					4\lambda \leq \abs y,
				\end{array}
			\right.
 \end{eqnarray}
 and

 \begin{eqnarray*}
 	\frac{\partial h_2}{\partial y_n}(y)
		& = &
		- a \Gamma_i^{-1} \delta_i \ell_i^{-1} \abs y^{-n} g_i(\abs y)
		\\
		& \leq &
		\left\{
			\begin{array}{ll}
				0
				&
				y\in \bdy'\Sigma_\lambda\intersect \overline B_{4\lambda}
				\\
				- a \Gamma_i^{-1} \delta_i \abs y^{-n}
				\left(
					\ell_i^{-1} \log \frac{\abs y}{\lambda}
					+
					\sum_{j = 2}^{n -3} \ell_i^{-j} \abs y^{j-1}
				\right)
				&
				y\in \bdy'\Sigma_\lambda \setminus B_{4\lambda},
			\end{array}
		\right.
 \end{eqnarray*}
 where $\bar a$ denotes a constant of the form $C(n)a$. Moreover, by elementary estimates we have $h_2(y) = \circ(1) \abs y^{2-n}$.

Set $h^\lambda = h_1 + h_2$. By the estimates of $h_1$ and $h_2$ we have $h^\lambda(y) = \circ(1) \abs y^{2-n}$ in $\Sigma_\lambda$. It remains to show that $h^\lambda$ satisfies \eqref{eq:Desired-Differential-Inequalities-Ti-Bounded}. Clearly, $w^\lambda + h^\lambda$ vanishes on $\bdy \Sigma_\lambda \intersect \bdy B_\lambda$, so we only need to show the differential inequalities in \eqref{eq:Desired-Differential-Inequalities-Ti-Bounded}. Since $h_2\leq 0$ and since each of $h_2$ and $\frac{\partial h_1}{\partial y_n}$ vanish on $\bdy'\Sigma_\lambda$, we have

 \begin{equation*}
 	\left\{
		\begin{array}{ll}
		 	L_i(w^\lambda + h^\lambda)(y)
				\leq (Q^\lambda - \Hat Q^\lambda)(y) + H_i(y - Re)\xi_1(y)h_1(y)
				+ \lap h_2(y)
			&
			y\in \Sigma_\lambda
			\\
			B_i(w^\lambda + h^\lambda)(y)
				= \abs{c_i}\xi_2(y) h_1(y) + \frac{\partial h_2}{\partial y_n}(y)
			&
			y\in \bdy'\Sigma_\lambda.
		\end{array}
	\right.
 \end{equation*}
For $y\in \overline \Sigma_\lambda \intersect \overline B_{3\lambda}$, each of $h_1$ and $\lap h_2$ are nonpositive so we have both $L_i(w^\lambda + h^\lambda)(y)\leq 0$ for $y\in \overline \Sigma_\lambda \intersect \overline B_{3\lambda}$ and $B_i(w^\lambda + h^\lambda)(y)\leq 0$ for $y\in \bdy'\Sigma_\lambda \intersect \overline B_{3\lambda}$. For $y\in (\overline \Sigma_\lambda \intersect \overline B_{4\lambda}) \setminus B_{3\lambda}$, $h_1\leq 0$. In addition, using both the estimates of $Q^\lambda - \Hat Q^\lambda$ in \eqref{eq:QLambda-QLambdaHat} and \eqref{eq:Laplacian-h2-Estimate-c<0}, since $\ell_i^{-1} = \circ(1)$, for any choice of $a$ we have

 \begin{equation*}
 	L_i(w^\lambda+ h^\lambda)(y)
		\leq
		C\Gamma_i^{-1} \delta_i \lambda^{-1-n}(\bar a M \ell_i^{-1} \lambda^2 - a_2)
		\leq
		0
		\qquad
		y\in (\Sigma_\lambda \intersect \overline B_{4\lambda})\setminus B_{3\lambda}
 \end{equation*}
provided $i$ is sufficiently large. Moreover, since each of $h_1$ and $\frac{\partial h_2}{\partial y_n}$ are nonpositive for $\abs y \leq 4\lambda$ we have $B_i(w^\lambda  + h^\lambda)(y) \leq 0$ for $y\in \bdy'\Sigma_\lambda \intersect \overline B_{4\lambda}$. Finally, if $\abs y \geq 4\lambda$ we must account for the possibility that $h_1\geq 0$. By construction of $h_2$ and the estimates of $\xi_2$ and $h_1$ given in Lemma \ref{lem:Mean-Value-Coefficient-Estimates-Ti-Unbounded} and \eqref{eq:Fast-h1-Estimates-Ti-Bounded} respectively, after choosing $a(n,\Lambda)$ sufficiently large, we have

 \begin{equation*}
 	B_i(w^\lambda + h^\lambda)(y)
		\leq
		C(\abs{c_i} - \bar a) \Gamma_i^{-1}\delta_i \abs y^{-n}
		\left(	
			\ell_i^{-1} \log \frac{\abs y}{\lambda}
			+
			\sum_{j = 2}^{n - 3} \ell_i^{-j} \abs y^{j-1}
		\right)
		\leq
		0
		\qquad
		y\in (\bdy'\Sigma_\lambda \intersect \overline{\mathcal O}_\lambda)\setminus B_{4\lambda}.
 \end{equation*}
 For the interior inequality we have

 \begin{eqnarray*}
 	L_i(w^\lambda + h^\lambda)(y)
		& \leq &
		\Gamma_i^{-1}\delta_i \abs y^{-2-n}
		\left(
			- \bar a_2 \abs y
			+ C_1\ell_i^{-1} \left( \log\frac{\abs y}{\lambda} + \bar a\abs y - \bar a_2\abs y^2\right)
			\phantom{\sum_{j = 2}^{n - 3}}
		\right.
		\\
		& &
		\left.
			+ C_2 \sum_{j = 2}^{n - 3}\ell_i^{-j} \abs y^{j - 1}
			(1 + \bar a \abs y - \bar a_2 \abs y^2)
		\right)
		\qquad
		y\in (\Sigma_\lambda \intersect \mathcal O_\lambda) \setminus B_{4\lambda}.
 \end{eqnarray*}
 Therefore, by choosing $R = R(a, a_2)$ larger if necessary, we have $L_i(w^\lambda + h^\lambda)(y) \leq 0$ for $y\in (\Sigma_\lambda \intersect \mathcal O_\lambda)\setminus B_{4\lambda}$. Estimates \eqref{eq:hLambda-o(1)} and \eqref{eq:Desired-Differential-Inequalities-Ti-Bounded} are satisfied in the case $c_0<0$. \\
 \emph{Case 2:} $c_0\geq 0$. \\
 In this case, either $c_i>0$ or $0\leq -c_i = \circ(1)$. For this case we set

 \begin{equation*}
 	g_i(r)
		=
		\left\{
			\begin{array}{ll}
				\lambda^{1-n} r^n - \frac{n - 1}{2\lambda} r^2 + \frac{\lambda}{2} (n - 3)
				&
				\lambda \leq r\leq 3\lambda
				\\
				\text{ smooth positive connection }
				&
				3\lambda \leq r\leq 4\lambda
				\\
				\lambda^{1+n} + \ell_i^{-1} \log \frac{\abs y}{\lambda}
				+ \sum_{j = 2}^{n - 3} \ell_i^{-j} \abs y^{j - 1}
				&
				4\lambda \leq r,
			\end{array}
		\right.
 \end{equation*}
 where `smooth positive connection' means there is an $i$-independent constant $M(\lambda)>0$ such that both $g_i(r) \geq M^{-1}$ for $3\lambda \leq r\leq 4\lambda$ and $\norm{g_i}_{C^2([3\lambda, 4\lambda])}\leq M$. Since $g_i(\lambda) = 0$, $g_i(3\lambda) = C\lambda$ and $g_i''(r)>0$ for $\lambda \leq r\leq 3\lambda$, there is a constant $C>0$ such that $g_i(r) \geq C(r - \lambda)$ for $\lambda \leq r\leq 3\lambda$. Moreover, by direct computation and elementary estimates we have

 \begin{equation*}
 	g_i''(r) - \frac{n - 1}{r}g_i'(r)
		\geq
		\left\{
			\begin{array}{ll}
				0
				&
				\lambda \leq r\leq 3\lambda
				\\
				-M
				&
				3\lambda \leq r\leq 4\lambda
				\\
				- C\sum_{j = 1}^{n - 3}\ell_i^{-j} r^{j - 3}
				&
				4\lambda \leq r.
			\end{array}
		\right.
 \end{equation*}
 Now set

 \begin{equation*}
 	h_2(y)
		=
		- a \Gamma_i^{-1} \delta_i y_n \abs y^{-n} g_i(\abs y)
		\leq
		0
		\qquad \abs y \geq \lambda.
 \end{equation*}
 By direct computation and elementary estimates we have both

 \begin{eqnarray*}
 	\lap h_2(y)
		& = &
		- a\Gamma_i^{-1} \delta_i y_n \abs y^{-n}
			\left( g_i''(\abs y) - \frac{n - 1}{\abs y} g_i'(\abs y)\right)
		\\
		&\leq &
			\left\{
				\begin{array}{ll}
					0
					&
					\lambda \leq \abs y \leq 3\lambda
					\\
					\bar a M \Gamma_i^{-1} \delta_i \lambda^{1-n}
					&
					3\lambda \leq \abs y \leq 4\lambda
					\\
					\bar a \Gamma_i^{-1} \delta_i\abs y^{-2-n}
					\sum_{j = 1}^{n - 3} \ell_i^{-j} \abs y^j
					&
					4\lambda \leq \abs y
				\end{array}
			\right.
 \end{eqnarray*}
 and

 \begin{eqnarray}\label{eq:h2-Boundary-Derivative-Estimate-c>0}
 	\frac{\partial h_2}{\partial y_n}(y)
		& = &
		- a \Gamma_i^{-1} \delta_i \abs y^{-n} g_i(\abs y)
		\\
		&\leq &
		\left\{
			\begin{array}{ll}
				- \bar a \Gamma_i^{-1} \delta_i \lambda^{-n}(\abs y - \lambda)
				&
				y\in \bdy'\Sigma_\lambda \intersect \overline B_{3\lambda}
				\\
				- \bar a  \Gamma_i^{-1} \delta_i \lambda^{-n} M^{-1}
				&
				y\in( \bdy'\Sigma_\lambda \intersect \overline B_{4\lambda})
					\setminus B_{3\lambda}
				\\
				- \bar a \Gamma_i^{-1} \delta_i \abs y^{-n}
				\left(
					\lambda^{1+n} + \ell_i^{-1} \log \frac{\abs y}{\lambda}
					+
					\sum_{j = 2}^{n - 3} \ell_i^{-j} \abs y^{j - 1}
				\right)
				&
				y\in \bdy'\Sigma_\lambda \setminus B_{4\lambda},
			\end{array}
		\right.
		\notag
 \end{eqnarray}
where $\bar a $ denotes a constant of the form $Ca$.

Set $h^\lambda = h_1 + h_2$. Then $h^\lambda(y) = \circ(1) \abs y^{2-n}$ and $h^\lambda = 0$ on $\bdy \Sigma_\lambda \intersect \bdy B_\lambda$. We need to show that the differential inequalities in \eqref{eq:Desired-Differential-Inequalities-Ti-Bounded} hold so we consider

\begin{equation*}
	\left\{
		\begin{array}{ll}
			L_i(w^\lambda + h^\lambda)(y)
				\leq
				(Q^\lambda - \Hat Q^\lambda)(y)
				+
				H_i(y - Re)\xi_1(y) h_1(y)
				+
				\lap h_2
			&
			y\in \Sigma_\lambda
			\\
			B_i(w^\lambda + h^\lambda)(y)
				\leq
				\abs{c_i}\xi_2(y)\abs{h_1(y)} + \frac{\partial h_2}{\partial y_n}(y)
			&
			y\in \bdy'\Sigma_\lambda.
		\end{array}
	\right.
\end{equation*}
For $y\in \overline \Sigma_\lambda \intersect \overline B_{3\lambda}$ both of $h_1$ and $\lap h_2$ are nonpositive, so $L_i(w^\lambda+ h^\lambda)(y) \leq 0$ on this set. Moreover, in view of \eqref{eq:Fast-Abs-hLambda-Estimate-Ti-Unbounded} and \eqref{eq:h2-Boundary-Derivative-Estimate-c>0}, once $a$ is chosen we may choose $\epsilon>0$ depending on $n,\Lambda, \lambda,a$ such that

\begin{equation*}
	B_i(w^\lambda + h^\lambda)(y)
		\leq
		C\Gamma_i^{-1} \delta_i(\abs y - \lambda) (\abs {c_i} - \lambda^{-n} \bar a)
		\leq
		0
		\qquad
		y\in \bdy'\Sigma_\lambda\intersect \overline B_{3\lambda}
\end{equation*}
whenever $c_i <\epsilon$. For $y\in (\Sigma_\lambda \intersect \overline B_{4\lambda})\setminus B_{3\lambda}$, by choosing $a(n,\Lambda M, \lambda, a_2)$ small we have

\begin{equation*}
	L_i(w^\lambda + h^\lambda)(y)
		\leq
		C\Gamma_i^{-1} \delta_i \lambda^{-1-n}(\bar a M \lambda^2  - a_2)
		\leq
		0
		\qquad
		y\in (\Sigma_\lambda \intersect \overline B_{4\lambda})\setminus B_{3\lambda}.
\end{equation*}
For $y\in (\bdy'\Sigma_\lambda \intersect \overline B_{4\lambda})\setminus B_{3\lambda}$, by decreasing $\epsilon$ if necessary we have

\begin{equation*}
	B_i(w^\lambda + h^\lambda)(y)
		\leq
		C\Gamma_i^{-1} \delta_i \lambda(\abs{c_i} -\lambda^{-n}\bar a)
		\leq
		0
		\qquad
		y\in (\bdy'\Sigma_\lambda \intersect \overline B_{4\lambda})\setminus B_{3\lambda}.
\end{equation*}
Finally, for $y\in \Sigma_\lambda \setminus B_{4\lambda}$ by choosing $R$ larger if necessary we have

\begin{eqnarray*}
	L_i(w^\lambda + h^\lambda)(y)
		& \leq &
		C\Gamma_i^{-1} \delta_i \abs y^{-2-n}
		\left(
			- a_2\abs y
			+
			\ell_i\left(C\log \frac{\abs y}{\lambda} + \bar a \abs y - a_2 \abs y^2\right)
			\phantom{\sum_{j = 2}^{n -3}}
		\right.
		\\
		& &
		\left.
			+
			\sum_{j = 2}^{n -3} \ell_i^{-j} \abs y^{j -1}(C + \bar a \abs y - a_2\abs y^2)
		\right)
		\\
		& \leq &
		0
\end{eqnarray*}
The boundary inequality for $\abs y\geq 4\lambda$ is

\begin{equation*}
	B_i(w^\lambda + h^\lambda)(y)
		\leq
		C\Gamma_i^{-1} \delta_i(\abs{c_i} - \bar a) \abs y^{-n}g_i(\abs y)
		\leq
		0
		\qquad
		y\in (\bdy'\Sigma_\lambda \intersect \mathcal O_\lambda)\setminus B_{4\lambda}.
\end{equation*}
We have shown that $h^\lambda$ satisfies \eqref{eq:Desired-Differential-Inequalities-Ti-Bounded} when $c_0\geq 0$.

Arguing as in the proof of Lemma \ref{lem:Start-Moving-Spheres-Ti-Bounded} shows that the moving sphere process can start at $\lambda = \lambda_0$. Then arguing as in the proof of Proposition \ref{lem:Vanishing-Ti-Bounded} we obtain a contradiction to the convergence of $v_{R,i}$ to $U_R$. Proposition \ref{lem:Fast-Vanishing-Ti-Bounded} is established.
\end{proof}
	
%
\subsection{Completion of the Proof of Theorem \ref{thm:Main}}
\label{subsection:Completion-of-Proof-of-Theorem}
With a rapid vanishing rate for $\delta_i$ in hand, a final application of the method of moving spheres will prove theorem \ref{thm:Main}. The rapid vanishing rate of $\delta_i$ makes the construction of the test function simple.

\begin{proof}[Proof of Theorem \ref{thm:Main}]
we consider $v_i$, $U$ and their Kelvin inversions

\begin{equation*}
	v_i^\lambda(y)
		=
		\left(\frac{\lambda}{\abs y}\right)^{n - 2}
		v_i(y^\lambda)
	\qquad
	\text{ and }
	\qquad
	U^\lambda(y)
		=
		\left(\frac{\lambda}{\abs y}\right)^{n - 2}
		U(y^\lambda)
\end{equation*}
for $y\in \Sigma_\lambda = \Omega_i\setminus \overline B_\lambda$. In this case, with $\lambda^* = 1$ direct computation yields

\begin{equation*}
	\left\{
		\begin{array}{lll}
			(U- U^\lambda)(y) >0
			&
			y\in \bb R^n \setminus \overline B_\lambda
			&
			\text{ if } \lambda <\lambda^*
			\\
			(U- U^\lambda)(y) <0
			&
			y\in \bb R^n \setminus \overline B_\lambda
			&
			\text{ if } \lambda >\lambda^*,
		\end{array}
	\right.
\end{equation*}
and we consider $\lambda$ between $\lambda_0 = 1/2$ and $\lambda_1 = 2$. Set

\begin{equation*}
	w^\lambda(y)
		=
		v_i(y) - v_i^\lambda(y)
		\qquad
		y\in \Sigma_\lambda.
\end{equation*}
Then $w^\lambda$ satisfies equations \eqref{eq:wLambda-Equations-Ti-Unbounded} - \eqref{eq:Interior-Error-Function} with $R= 0$. We still need to construct a test function $h^\lambda$ such that \eqref{eq:Perturbation-Test-Function} and \eqref{eq:Desired-Differential-Inequalities} hold. Note that \eqref{eq:viLambda-Upper-Bound} still holds. By Proposition \ref{lem:Fast-Vanishing-Ti-Bounded} and  \eqref{eq:viLambda-Upper-Bound} we have

\begin{equation}\label{eq:Final-Proof-QLambda-Estimate}
	Q^\lambda(y)
		\leq
		C \Gamma_i^{2-n} \abs y^{-4}
		\qquad
		y\in \Sigma_\lambda. 	
\end{equation}
Let
\begin{equation*}
	h_1(y) = - a_1 \Gamma_i^{2-n} \lambda^{n + 2} (\lambda^{-1}- \abs y^{-1})
	\qquad \abs y \geq \lambda,
\end{equation*}
where $a_1$ is a positive constant which is to be determined. Routine computations show that $h_1$ satisfies

\begin{equation}\label{eq:h1BasicProperties}
	\left\{
		\begin{array}{ll}
			\lap h_1 \leq - a_1\Gamma_i^{2-n} \lambda^{n+2} \abs y^{-3}
			& y\in \Sigma_\lambda\\
			\frac{\partial h_1}{\partial y_n} = 0
			& y\in \bdy \Sigma_\lambda \intersect \bdy \bb R_+^n\\
			h_1(y) = 0
			& y\in \bdy \Sigma_\lambda \intersect \bdy B_\lambda.
		\end{array}
	\right.
\end{equation}
Moreover,

\begin{equation}\label{eq:Absh1Estimate}
	\abs{h_1(y)}
	\leq
	\left\{
		\begin{array}{ll}
			a_1\Gamma_i^{2-n} \lambda^n (\abs y- \lambda)
			&
			\lambda<\abs y \leq\lambda + 1
			\\
			a_1\Gamma_i^{2-n}\lambda^{n +1}
			&
			\lambda + 1\leq \abs y.
		\end{array}
	\right.
\end{equation}
In particular, $h_1(y) = \circ (1) \abs y^{2-n}$ for $\abs y \leq \epsilon_i \Gamma_i^{-1}$. \\
Next we define $h_2$. For $\lambda \leq r <\infty$, let $g(r)$ be a smooth positive function satisfying
\begin{equation*}
	g(r) =
	\left\{
		\begin{array}{ll}
			 r - \lambda + \frac{n - 1}{2\lambda}(r  - \lambda)^2
			 &
			 \lambda <r \leq 3\lambda
			 \\
			 \text{ smooth and positive }
			 &
			 3\lambda \leq r \leq 4\lambda
			 \\
			 \lambda^{-1}  - r^{-1}
			 &
			 4\lambda \leq r,
		\end{array}
	\right.
\end{equation*}
where `smooth and positive' means there exists a constant $M>0$ such that both

\begin{equation*}
	g(r) \geq \frac 1 M \qquad 3\lambda \leq r \leq 4\lambda
\end{equation*}
and

\begin{equation*}
	\norm{g}_{C^2([\lambda, \infty))}\leq M.
\end{equation*}
Define

\begin{equation*}
	h_2(y) = - a_2 \Gamma_i^{2-n} y_n \abs y^{-2} g(\abs y)\qquad y\in \Sigma_\lambda,
\end{equation*}
where $a_2$ is a positive constant to be determined. Routine computations yield
\begin{equation}\label{eq:h2BasicProperties}
	\left\{
		\begin{array}{ll}
			\lap h_2(y) = - a_2 \Gamma_i^{2-n} y_n \abs y^{-2}
			\left(
				g''(\abs y) + \frac{n - 3}{\abs y}g'(\abs y) - \frac{2(n - 2)}{\abs y^2} g(\abs y)
			\right)
			&
			y\in \bb R^n \setminus B_\lambda
			\\
			\frac{\partial h_2}{\partial y_n}(y)
				=
				- a_2 \Gamma_i^{2-n} \abs y^{-2} g(\abs y)
			&
			y\in \bdy \bb R_+^n\setminus B_{\lambda}
			\\
			h_2(y) = 0
			&
			y\in \bdy B_\lambda \union \bdy\bb R_+^n
			\\
			\abs{h_2(y)} = \circ(1) \abs y^{2-n}
			&
			\lambda \leq \abs y \leq \epsilon_i\Gamma_i
		\end{array}
	\right.
\end{equation}
Moreover,
\begin{eqnarray*}
	\abs{h_2(y)}
		& \leq &
		C_2 u_i(x_i)^{-2}\abs y^{-1} g(\abs y)\\
		&\leq  &
		C_2 Mu_i(x_i)^{-2} \abs y^{-1}\\
		& = &
		\circ(1) \abs y^{2-n}.
\end{eqnarray*}
Set $h^\lambda(y) = h_1(y) + h_2(y)$.
Since each of $h_1$ and $h_2$ are nonpositive in $\Sigma_\lambda$, using \eqref{eq:h1BasicProperties}, \eqref{eq:Final-Proof-QLambda-Estimate} and \eqref{eq:h2BasicProperties}, we see that $a_1$ may be chosen sufficiently large and depending on $a_2$ such that $L_i(w^\lambda + h^\lambda)(y) \leq 0$ in $\Sigma_\lambda$. Now, if $c_i\leq 0$, then $B_i(w^\lambda + h^\lambda)\leq 0$ in $\bdy'\Sigma_\lambda \intersect \overline B_{4\lambda}$ holds trivially. If $c_i>0$, then using the estimates for $\abs{h_1}$ and $\abs{h_2}$ along with lemma \ref{lem:Mean-Value-Coefficient-Estimates-Ti-Unbounded} and \eqref{eq:h2BasicProperties}, we see that there is $0<\epsilon = \epsilon(\lambda, a_1, a_2)$ such that if $c_i<\epsilon$ then $B_i(w^\lambda + h^\lambda)\leq 0$ on $(\bdy'\Sigma_\lambda \intersect \overline{\mathcal O}_\lambda)\setminus B_{4\lambda}$. Finally, arguing similarly to the proof of lemma \ref{lem:Moving-Spheres-Can-Start} we see that the moving-sphere process can start at $\lambda_0 = 1/2$. Then arguing as in the proof of lemma \ref{lem:Vanishing-Ti-Unbounded}, we see that the spheres can be moved to $\lambda_1 = 2$, which is a contradiction. Theorem \ref{thm:Main} is established.
\end{proof}

\section{Energy Estimate}
	\label{section:Energy-Estimate}

In this section we give an overview of the proof of Corollary \ref{coro:Energy-Estimate}. The major step in the proof is the derivation of the Harnack-type inequality. Since the proof of Corollary \ref{coro:Energy-Estimate} is standard once Thoerem \ref{thm:Main} is obtained (see \cite{Li1995}, \cite{HanLi1999} and \cite{LiZhang2003} for details), only the key points of the proof will be mentioned here. \\

First, use the selection process of Schoen to locate all large local maximums of $u$ in $B_{2/3}^+$. Surrounding each local maximizer of $u$, there is a neighborhood in which $u$ is well-approximated by a standard bubble, the majority of whose energy is in this neighborhood. The key information revealed by the Harnack-type inequality is that the distance between the local maximizers of $u$ is not too small.\\

Due to the local nature of the equations considered in this article, the approach in controlling this distance between maximizers of $u$ is slightly different than the approach used in \cite{Li1995} so we mention it now. For the local equations, it is not possible to find two local maximizers of $u$ that are mutually closest to each other. Each local maximizer certainly has a second maximizer which is closest to it, but there may be a third local maximizer whose distance to the second local maximizer is smaller than the distance from the first local maximizer to the second local maximizer. To overcome this difficulty, rescale the equation so that the distance from the first local maximizer to the nearest local maximizer is one. The Harnack-type inequality forces the values of $u$ at these two local maximum points to be comparable. The comparability of these two maximum values ensures that no two bubbles can tend to the same blow-up point. Indeed, if two bubbles tend to the same blow-up point, then a harmonic function with positive second-order term can be constructed. This function will give a contradiction in the Pohozaev identity. \\

With the distance between local maximizers of $u$ controlled, one can use standard elliptic theory to show that near a large local maximum, $u$ behaves like a rapidly decaying harmonic function. This behavior yields the energy estimate in Corollary \ref{coro:Energy-Estimate}.

\section{Appendix}
		
\subsection{Green's Function Estimates}\label{subsection:Greens-Estimates-Appendix}

\begin{lem}\label{lem:GreensEstimates}
	Let
	\begin{equation*}
			\begin{array}{l}
				A = \{\eta\in \Sigma_\lambda: \abs{ y - \eta}\leq (\abs y - \lambda)/3\}\\
				B = \{\eta\in \Sigma_\lambda: \abs{ y - \eta}\geq (\abs y - \lambda)/3\text{ and }
					\abs \eta \leq 8\lambda\}\\
				D = \{ \eta \in \Sigma_\lambda: \abs \eta \geq 8\lambda\}.
			\end{array}
	\end{equation*}
	There exist positive constants $C_1$ and $C_2$ depending only on $n$ such that the following estimates hold.
	\begin{enumerate}
		\item For all $\lambda< \abs y \leq 4\lambda$,
			\begin{equation}\label{eq:GLowerBoundySmall}
				G^\lambda(y, \eta)
					\geq
					C_1\frac{(\abs y - \lambda)(\abs \eta - \lambda)}{\lambda^n}
				\qquad
				\eta\in \Omega_\lambda
			\end{equation}
			and
			\begin{equation}\label{eq:GUpperBoundySmall}
			G^\lambda(y, \eta) \leq
				\left\{
					\begin{array}{ll}
						C\abs{y - \eta}^{2-n}
						&
						\eta\in A
						\\
						C
						\frac{(\abs y - \lambda)(\abs \eta^2 - \lambda^2)}
							{\lambda\abs{y - \eta}^n}
						\leq
						C\frac{(\abs y - \lambda)(\abs \eta- \lambda)}{\abs{y - \eta}^n}
						&
						\eta \in B
						\\
						C\frac{(\abs y - \lambda)(\abs\eta^2 - \lambda^2)}
							{\lambda \abs{y - \eta}^n}
						\leq
						C\frac{\abs y - \lambda}{\lambda} \abs\eta^{2-n}
						&
						\eta\in D
					\end{array}
				\right.
			\end{equation}
		\item For all $\abs y \geq 4\lambda$, both
			\begin{equation}\label{eq:GLowerBoundyLarge}
				G^\lambda(y, \eta)
					\geq
						C\frac{(\abs \eta - \lambda)(\abs y^2 - \lambda^2)}
						{\lambda\abs{y  - \eta}^n}
					\geq
						C\frac{\abs \eta - \lambda}{\lambda} \abs y^{2-n}
					\qquad
					\eta\in \Omega_\lambda
			\end{equation}
			and
			\begin{equation}\label{eq:GUpperBoundyLarge}
				G^\lambda(y, \eta)
					\leq
					C\abs{y - \eta}^{2-n}
					\qquad
					\eta\in \Sigma_\lambda\setminus \Omega_\lambda.
			\end{equation}
	\end{enumerate}
\end{lem}
\begin{proof}
	By \eqref{eq:Greens-Function-Formula} after preforming elementary computations involving the mean-value theorem we obtain
	
	\begin{equation*}
		0
			\leq
			\sigma_n G(y, \eta)
			=
			\frac{1}{2\lambda^2} (\abs y^2 - \lambda^2)(\abs \eta^2 - \lambda^2)
			\int_0^1 \ell_t(y, \eta)^{-\frac n 2}\; dt,
	\end{equation*}
	where
	
	\begin{eqnarray*}
		\ell_t(y, \eta)
			& = &
			t\abs{y - \eta}^2
			+ (1 - t) \left(\frac{\abs y}{\lambda}\right)^2 \abs{ y^\lambda - \eta}^2
			\\
			& =&
			\left(\frac{\abs y}{\lambda}\right)^2 \abs{y^\lambda- \eta}^2
			-
			\frac{t}{\lambda^2} (\abs y^2 - \lambda^2)(\abs \eta^2 - \lambda^2)
			\qquad
			0\leq t\leq 1; \;
			(y, \eta) \in \Sigma_\lambda\times \Sigma_\lambda\setminus \{y = \eta\}.
	\end{eqnarray*}
	For each $(y, \eta)\in \Sigma_\lambda \times \Sigma_\lambda \setminus \{y = \eta\}$, $t\mapsto \ell_t(y, \eta)$ is decreasing and positive, so for such $(y, \eta)$,
	
	\begin{eqnarray}\label{eq:Greens-Function-Mean-Value-Estimates}
		\frac{1}{2\lambda^2} (\abs y^2 - \lambda^2)(\abs \eta^2 - \lambda^2)
			\left(\frac{\abs y}{\lambda}\right)^{- n} \abs{y^\lambda - \eta}^{-n}
			&\leq &
			\sigma_n G(y, \eta)
			\\
			& \leq &
			\frac{1}{2\lambda^2} (\abs y^2 - \lambda^2)(\abs \eta^2 - \lambda^2)
			\abs{ y - \eta}^{-n}.\notag 	
	\end{eqnarray}
	Each of the estimates in \eqref{eq:GLowerBoundySmall}, \eqref{eq:GUpperBoundySmall} and \eqref{eq:GUpperBoundyLarge} follow immediately from either \eqref{eq:Greens-Function-Formula} or from \eqref{eq:Greens-Function-Mean-Value-Estimates}. To show $G(y, \eta)$ satisfies \eqref{eq:GLowerBoundyLarge}, use \eqref{eq:Greens-Function-Mean-Value-Estimates} in addition to the fact that $G(y, \eta) = G(\eta, y)$. \\
	
	To see that \eqref{eq:GLowerBoundySmall}, \eqref{eq:GUpperBoundySmall}, \eqref{eq:GLowerBoundyLarge} and \eqref{eq:GUpperBoundyLarge} hold for $\overline G$, observe that since $G(y, \eta)\geq 0$, $\overline G(y, \eta)\geq G(y, \eta)$. This gives both \eqref{eq:GLowerBoundySmall} and \eqref{eq:GLowerBoundyLarge}. To show that $\overline G$ satisfies \eqref{eq:GUpperBoundySmall} and \eqref{eq:GUpperBoundyLarge}, observe that $G(\bar y, \eta)$ satisfies these inequalities with $y$ replaced by $\bar y$. Since $\abs {\bar y} = \abs y$ and $\abs{\bar y- \eta} \geq \abs{y - \eta}$ for $y, \eta\in \bb R_+^n$, the desired inequalities hold.
\end{proof}

\end{document}